\documentclass[letterpaper,11pt]{article}
\usepackage[utf8]{inputenc}
\usepackage[english]{babel}
\usepackage[a4paper,left=1in,right=1in,top=1in,bottom=1in]{geometry}
\usepackage{xspace}
\usepackage{amssymb}
\usepackage[all,web]{xy}
\usepackage{textcomp}
\usepackage{hyperref}
\usepackage{amsthm}
\usepackage{amsmath}
\usepackage{amsfonts}
\usepackage[shortlabels]{enumitem}
\usepackage{pgf,tikz}
\usepackage{tensor}
\usetikzlibrary{arrows,arrows.meta,bending,calc,positioning,shapes.geometric,shapes.symbols,shapes.misc,decorations.markings}
\usepackage{graphicx}
\usepackage{cite}
\usepackage{xcolor}
\usepackage{verbatim}
\usepackage{float}
\restylefloat{table}
\usepackage[strict]{changepage}

\usepackage{pifont}
\newtheorem*{thm*}{Theorem}
\newtheorem{thm}{Theorem}[section]
\newtheorem{lem}[thm]{Lemma}

\newtheorem*{conj*}{Conjecture}

\theoremstyle{definition}
\theoremstyle{remark}

\usetikzlibrary{calc,positioning,shapes.geometric,shapes.symbols,shapes.misc,decorations.markings}

\DeclareMathOperator{\Span}{span}

\DeclareMathOperator{\End}{End}

\title{Combinatorics of Vogan diagrams for almost-K\"ahler manifolds}
\author{Alice Gatti}
\date{}
\begin{document}
\newcommand{\lrar}[1]{\langle #1 \rangle}
\newcommand{\lrbr}[1]{\lbrace #1 \rbrace}
\newcommand{\lrbrabs}[1]{\lvert #1\rvert}
\newcommand{\ZZ}{\mathbb{Z}}
\newcommand{\RR}{\mathbb{R}}
\newcommand{\NN}{\mathbb{N}} 
\newcommand{\CC}{\mathbb{C}}
\newcommand{\HH}{\mathbb{H}}
\newcommand{\CP}{\mathbb{CP}}
\newcommand{\OO}{\mathbb{O}}
\newcommand{\dime}{\mathrm{dim}}
\newcommand{\tr}{\mathrm{tr}}
\newcommand{\de}{\partial}
\newcommand{\lie}{\mathcal{L}}
\newcommand{\dde}[1]{\frac{\partial}{\partial#1}}
\newcommand{\JComp}{\mathcal{J}}
\newcommand{\lgg}{\mathfrak{g}}
\newcommand{\mm}{\mathfrak{m}}
\newcommand{\kk}{\mathfrak{k}}
\newcommand{\ad}{\mathrm{ad}}
\newcommand{\lggc}{\mathfrak{g}_{\CC}}
\newcommand{\pp}{\mathfrak{p}}
\newcommand{\hh}{\mathfrak{h}}
\newcommand{\ep}[1]{\varepsilon_{#1}}
\newcommand{\spann}{\mathrm{span}}
\newcommand{\ttt}{\mathfrak{t}}
\newcommand{\aaa}{\mathfrak{a}}
\newcommand{\dd}{\mathrm{d}}

\allowdisplaybreaks

\maketitle

\begin{abstract}
Let $G$ be a non-compact classical semisimple Lie group and let $G/V$ be the adjoint orbit with respect to a fixed element in $G$. These manifolds can be equipped with an almost-K\"ahler structure and we provide explicit formulae for the existence of special almost-complex structures on $G/V$ purely in terms of the combinatorics of the associated Vogan diagram. The formulae are given separately for Lie groups whose Lie algebras are of type $A_{\ell}$, $B_{\ell}$, $C_{\ell}$, $D_{\ell}$, where $\ell$ denotes the rank of the Lie algebra.
\end{abstract}


\section{Introduction}

The aim of this paper is to highlight the combinatorial structure of a problem coming from symplectic geometry, that is, finding almost-complex structures on a given symplectic manifold satisfying a geometric and differential property. We will restrict to a particular class of symplectic manifolds that are called \emph{adjoint orbits}, that are diffeomorphic to a quotient $G/V$ where $G$ is a semisimple Lie group and $V$ a compact subgroup. On these manifolds one can choose a symplectic structure and an almost-complex structure that are \emph{invariant} by the group action of $G$ on $G/V$, meaning that these structures become algebraic objects on the Lie algebra $\mathfrak{g}$ of $G$. By the assumption of $G$ being semisimple, the symplectic and almost-complex structures can be studied in a purely Lie-theoretical way. In turn, also the differential properties of interest concerning the almost-complex structure can be formulated in a Lie theory framework. This is the setting studied in \cite{DellaVedovaGatti2022}, where the authors formulated the condition in terms of the root system of the underlying Lie algebra $\mathfrak{g}$. In this work, we take a step further and we rephrase the condition solely in terms of the combinatorics of the Vogan diagram of $\mathfrak{g}$.

The problem is explained with more details in the following. Let $(M,\omega)$ be a symplectic manifold, and let $J$ be a compatible almost-complex structure on it. This means that $J$ is an endomorphism of the tangent bundle of $M$, $TM$, such that $J^2=$id and satisfying $\omega(JX,JY)=\omega(X,Y)$ and $\omega(JX,X)>0$, for all vector fields $X,Y\in TM$. Manifolds equipped with these structures are called \emph{almost-K\"ahler manifolds}. When, in addition, $J$ is integrable, such manifolds turn out to be complex manifolds and are called \emph{K\"ahler manifolds}. One can define a closed two-form $\rho$ on $M$ associated with $J$ in the following way. Let $\nabla$ denote the Chern connection of $J$. The curvature $R$ of $\nabla$ is a two-form with values in $\End(TM)$, and $\rho$ is defined by the identity $\rho(X,Y) = \tr(JR(X,Y))$, for $X,Y$ vector fields on $M$. In particular, an open problem is to find when the form $\rho$ is a multiple of the symplectic form $\omega$, i.e.,
\begin{equation}\label{eq:rholambdaomega}	
\rho=\lambda\omega,\quad \lambda\in\RR.
\end{equation}
This condition has been extensively studied in several works, especially on K\"ahler manifolds \cite{Calabi1957,Yau1977,Yau1978,DonaldsonChenSun2015I,DonaldsonChenSun2015II,DonaldsonChenSun2015III,Tian1997,Tian1990,Tian2015}, just to name a few references. In this setting, when the condition \eqref{eq:rholambdaomega} is satisfied for some $\lambda\in\RR$, the manifold is said to be \emph{K\"ahler-Einstein}. On the other hand, very little is known about equation \eqref{eq:rholambdaomega} for almost-K\"ahler manifolds. Hence it turns out to be important to look for examples of symplectic non-complex manifolds satisfying this equation. Almost-K\"ahler manifolds with this property will be called \emph{special}.

In this work, we focus on manifolds $M$ that are \emph{homogeneous spaces}, i.e., diffeomorphic to a quotient $G/V$ of a Lie group $G$ by a subgroup $V$. On such manifolds one can choose a symplectic and an almost-complex structure that are invariant by the action of $G$. In this case the condition $\rho=\lambda\omega$ becomes an algebraic condition on the Lie algebra $\mathfrak{g}$ of the group $G$, which greatly simplify the computations. Moreover, when $G$ is assumed to be a real non-compact semisimple Lie group and $V$ a compact subgroup, it turns out that the form $\rho$ can be expressed entirely in terms of the root system of the real semisimple Lie algebra $\mathfrak{g}$. It is possible to list all classical adjoint orbits satisfying \eqref{eq:rholambdaomega} having Lie algebra with rank up to $4$, and all the exceptional ones \cite{DellaVedovaGatti2022}. In particular, adjoint orbits $G/V$ with $G$ an exceptional simple Lie group satisfying \eqref{eq:rholambdaomega} are fully classified. A key ingredient in proving these results are Vogan diagrams \cite[Ch.~VI,Sec.~8]{Knapp1996}, which are combinatorial tools akin to Dynkin diagrams for non-compact real semisimple Lie algebras. Through Vogan diagrams it is possible express the equation \eqref{eq:rholambdaomega} on adjoint orbits of semisimple Lie groups in terms of the root system associated to the real non-compact semisimple Lie algebra $\mathfrak{g}$. In this work we write the equation only in terms of the indices of simple non-compact roots of the Vogan diagram. The results are quite technical, so we briefly summarize the content below.

\begin{thm}
Given a classical Vogan diagram of rank $\ell$ with $S$ the set of indices of non-compact simple roots, we provide explicit formulae to determine when equation $\rho=\lambda\omega$ is satisfied. These formulae depend only on the combinatorics of the set $S$.
\end{thm}

The results that build the above theorem also shed some light on the combinatorics of Vogan diagrams and how to count certain compact roots in real non-compact semisimple Lie algebras. In addition, they allow to implement a faster algorithm to verify when equation \eqref{eq:rholambdaomega} is satisfied on these orbits for some $\lambda\in\RR$. 
We will prove the results for classical simple Lie algebras only, since these can be trivially adapted to semisimple Lie algebras. In general it would be very interesting to find a unifying way of writing the combinatorial formulae, since the current results require to consider each family of classical Lie algebras $A$, $B$, $C$ and $D$ separately. For example, one can define an operator acting on Vogan diagrams or particular subgraphs and analyzing its properties, like its spectrum. We leave this study for future work.

The paper is organized as follows. In section \ref{sec:specialAKadjointorbits} we recall the relevant background material presented in \cite{DellaVedovaGatti2022}, omitting all the proofs. In section \ref{sec:specialCVD} we prove the results concerning equation \eqref{eq:rholambdaomega} and the combinatorics of Vogan diagrams. More precisely, we give explicit formulae to check if the equation \eqref{eq:rholambdaomega} holds with purely combinatorial arguments, for each family of classical Lie algebras $A$ \ref{sec:anCoefficients}, $B$ \ref{sec:bnCoefficients}, $C$ \ref{sec:cnCoefficients}, $D$ \ref{sec:dnCoefficients}. 


\section{Special almost-K\"ahler adjoint orbits}\label{sec:specialAKadjointorbits}

In this section we recall definitions and results needed for the computations in section \ref{sec::spacialitycondition}. In particular, some facts about adjoint orbits of semisimple Lie groups, the structure theory of real semisimple Lie algebras and the algebraic equation \eqref{eq:rholambdaomega} for adjoint orbits of semisimple Lie groups. We will omit the proofs and the details of all the results, as they are already contained in \cite{DellaVedovaGatti2022}. Additional material on the general theory can be found in \cite{Knapp1996,Helgason1978,GriffithsSchmid1969,DellaVedova2019}.


\subsection{Adjoint orbits}\label{sec:adjointOrbits}

Let $G$ be a Lie group and denote by $\mathfrak g$ its Lie algebra. Then $G$ acts on $\mathfrak{g}$ by the \emph{adjoint action of} $G$ \emph{on} $\mathfrak{g}$
\begin{equation}
\text{Ad}:G\to \text{Aut}(\mathfrak{g}),\quad g\mapsto \text{Ad}_g,
\end{equation}
where $\text{Ad}_g$ is the differential at the identity $e\in G$ of the conjugation $h\to ghg^{-1}$, $\forall h\in G$. By differentiating the above map, one gets
\begin{equation}\label{eq:adjointaction}
\text{ad}:\mathfrak{g}\to\text{Der}(\mathfrak{g}),\quad x\mapsto\text{ad}_x,
\end{equation}
and $\text{ad}_x(y)=[x,y]$, where $[,]$ denotes the commutator and $\text{Der}(\mathfrak{g})$ is the Lie algebra of $\text{Aut}(\mathfrak{g})$. Also in this case we call the map \eqref{eq:adjointaction} \emph{adjoint action of} $\mathfrak{g}$ \emph{on itself}.

Assume that $G$ is, in addition, non-compact, real and semisimple. Let $v \in \mathfrak g$ be a chosen element such that its stabilizer $V$ is a compact subgroup $V \subset G$. Note that the orbit of $v$ under the adjoint action, also called \emph{adjoint orbit} of $v$, is diffeomorphic to $G/V$, by the orbit-stabilizer theorem. By definition of adjoint action, the Lie algebra of the stabilizer $V$ is 
\begin{equation}
\mathfrak v = \{x \in \mathfrak g \ \vert \ [v,x]=0 \}.
\end{equation}
The \emph{Killing form} $B(x,y) = \tr(\ad(x) \ad(y))$, $x,y \in \mathfrak g$, is a bilinear symmetric form defined on a Lie algebra $\mathfrak{g}$. In particular, since $V$ is compact, it holds that $B$ restricts to a negative definite scalar product on the Lie algebra $\mathfrak v$ of $V$ \cite[Ch.~VI, Sec.~1]{Knapp1996}, while the orthogonal complement
\begin{equation}
\mathfrak m = \{x \in \mathfrak g \ \vert \ B(x,y)=0 \mbox{ for all } y \in \mathfrak v\}
\end{equation}
is canonically isomorphic to the tangent space at the identity coset $e$ of the adjoint orbit $G/V$ of $v$.

Recall that $\mathfrak{g}$ is semisimple, thus the Killing form is non-degenerate, by Cartan criterion of semisemplicity \cite[Ch.~1, Sec.~7]{Knapp1996}. So $B$ induces a canonical isomorphism between $\mathfrak g$ and its dual $\mathfrak g^*$. 
As a consequence, $G/V$ turns out to be equipped with the so called \emph{Kirillov-Kostant-Souriau symplectic form} $\omega$ \cite[Sec.~1.2]{Kirillov2004}, which is $G$-invariant and, at the identity coset $e$, corresponds to the symplectic form (i.e., antisymmetric and non-degenerate) $\sigma$ on $\mathfrak m$ defined by
\begin{equation}
\sigma(x,y) = B(v,[x,y]) \quad x,y \in \mathfrak m.
\end{equation}

The relationship between the symplectic form $\omega$ on $G/V$ and $\sigma$ is analyzed in detail in \cite[Section 3]{DellaVedova2019}.


\subsection{Structure of the Lie algebra $\mathfrak g$}\label{sec::structureofg}

Given a complex Lie algebra $\mathfrak g_c$ obtained by complexification of $\mathfrak g$, there exists a unique complex conjugation $\tau$ on $\mathfrak{g}_c$ that fixes $\mathfrak g \subset \mathfrak g_c$. Let $\mathfrak k \subset \mathfrak g$ be a maximal compact subalgebra such that $\mathfrak v \subset \mathfrak k$. Then, considering the complexification $\mathfrak k_c \subset \mathfrak g_c$ of $\mathfrak k$ and its $\ad(\mathfrak k_c)$-invariant complement $\mathfrak p_c$ yields a decomposition $\mathfrak g = \mathfrak k \oplus \mathfrak p$, where $\mathfrak p = \mathfrak p_c \cap \mathfrak g$. This decomposition is called \emph{Cartan decomposition of} $\mathfrak{g}$.

Choose a maximal abelian subalgebra $\mathfrak{h}_0\subset\mathfrak{k}$ that contains $v$ and call $\mathfrak{h}_c\subset\mathfrak{g}_c$ its complexification. Then the adjoint representation of $\mathfrak h_c$ on $\mathfrak g_c$ produces a \emph{root space decomposition}
\begin{equation}
\mathfrak g_c = \mathfrak h_c \oplus \sum_{\alpha \in \Delta} \mathfrak g^\alpha,
\end{equation}
where the set of roots $\Delta$ is a finite subset of the dual space of $\mathfrak h_c$, and each root space,
\begin{equation}\label{eq::root-space}
\mathfrak g^\alpha = \left\{ x \in \mathfrak g_c \ \vert \ [h,x] = \alpha(h)x \mbox{ for all } h \in \mathfrak h_c \right\}
\end{equation}
has dimension one. For more details about the properties of root spaces and root space decompositions see \cite[Sec.~9.2]{Humphreys1972}. In particular, any root space $\mathfrak g^\alpha$ is contained either in $\mathfrak k_c$ or in $\mathfrak p_c$, and the root $\alpha$ is called \emph{compact} if its root space is contained in $\mathfrak k_c$ and \emph{non-compact} if it is contained in $\mathfrak p_c$.
Define the coefficients $\varepsilon_\alpha = -1$ if $\alpha$ is compact and $\varepsilon_\alpha = 1$ otherwise, for $\alpha\in \Delta$. In our case, it turns out to be convenient to choose always \emph{positive root systems}, where a positive root system is a subset $\Delta_+ \subset \Delta$ such that 
\begin{enumerate}\renewcommand{\labelenumi}{\alph{enumi})}
\item for all $\alpha \in \Delta$, either $\alpha$ or $-\alpha$ belongs to $\Delta_+$,
\item if $\alpha, \beta \in \Delta_+$ and $\alpha+\beta \in \Delta$, then $\alpha + \beta \in \Delta_+$.
\end{enumerate}
A positive root is called simple if it cannot be written as a sum $\alpha + \beta$ where $\alpha, \beta \in \Delta_+$. 
Once a positive root system $\Delta_+$ is fixed, the set of simple roots $\Sigma^+ \subset \Delta_+$ turns out to be a basis for $\mathfrak h_{\mathbf R}^*=(i\mathfrak{h}_0)^*$, since one can consider $\Delta$ as a subspace of $\mathfrak h_{\mathbf R}^*$. In addition, if $\alpha\in \Delta_+$ is a root, then it can be written as $\alpha = \sum_{\gamma \in \Sigma^+} n_\gamma \gamma$, with the coefficients $n_{\gamma}$ all positive integers.
Denote by $\Sigma^+_c = \{ \gamma \in \Sigma^+ \, | \, \varepsilon_\gamma = -1 \}$ the set of simple compact roots and by $\Sigma^+_n = \{ \gamma \in \Sigma^+ \, | \, \varepsilon_\gamma = 1 \}$ the set of simple non-compact roots. Then the set of simple roots can be decomposed as $\Sigma^+ = \Sigma^+_c \cup \Sigma^+_n$. It is possible to determine whether a root is compact or not by looking at the compactness of the simple roots $\gamma\in\Sigma^+$ and the coefficients $n_{\gamma}$'s, as the next result shows.
\begin{lem}\label{cor::epsilonalphaformula}
If a positive root $\alpha\in\Delta_+$ has the form $\alpha = \sum_{\gamma \in \Sigma^+} n_\gamma \gamma$, then 
\begin{equation}
\varepsilon_\alpha = (-1)^{1+\sum_{\gamma \in \Sigma^+_n}n_\gamma}.
\end{equation}
\end{lem}
Remember that $\mathfrak{h}_0\subset \mathfrak{k}$, with $\mathfrak{k}$ compact, hence $B$ restricts to a positive scalar product on $\mathfrak h_{\mathbf R}=i\mathfrak{h}_0$. 
As a consequence there is an isomorphism between $\mathfrak h_{\mathbf R}^*$ and $\mathfrak h_{\mathbf R}$ which takes $\psi\in\mathfrak{h}_{\mathbf{R}}^*$ to the unique $h_\psi\in\mathfrak{h}_{\mathbf{R}}$ such that $\psi(h) = B(h_\psi,h)$, for all $h \in \mathfrak h_{\mathbf R}$.
Thus, it is possible to define a positive scalar product on $\mathfrak h_{\mathbf R}^*$ as $(\psi,\psi') = B(h_{\psi},h_{\psi'})$.
This scalar product defined on $\mathfrak h_{\mathbf R}^*$ allows to define particular subspaces of $\mathfrak h_{\mathbf R}^*$, called \emph{Weyl chambers}, in the following way. Consider the set of hyperplanes $P_\alpha = \{ \psi \in \mathfrak h_{\mathbf R}^* \,|\, (\psi,\alpha)=0\}$, with $\alpha \in \Delta$. These hyperplanes divide $\mathfrak h_{\mathbf R}^*$ into a finite number of closed convex cones, named Weyl chambers. In particular, each positive root system $\Delta_+$ corresponds bijectively to a \emph{dominant Weyl chamber} defined by
\begin{equation}
C = \left\{ \psi \in \mathfrak h_{\mathbf R}^* \,|\, (\psi,\alpha) \geq 0 \mbox{ for all } \alpha \in \Delta_+ \right\}.
\end{equation}
	
Recall that our vector $v$ has been chosen to belong to $\mathfrak h_0$, hence $iv\in\mathfrak h_{\mathbf R}$. So, by the isomorphism between $\mathfrak h_{\mathbf R}$ and $\mathfrak h_{\mathbf R}^*$ recalled above, there exists a unique a co-vector $\varphi \in \mathfrak h_{\mathbf R}^*$ such that $h_\varphi = -iv$. In particular, one can always choose a positive root system $\Delta_+$ such that $\varphi$ belongs to the fundamental Weyl chamber $C$, and in the rest of the paper we assume such a choice of positive root system $\Delta_+$ has been made.


\subsection{Fundamental dominant weights}\label{ssec:fundamentalDominantWeights}
There exists a convenient basis of $C$ by means of fundamental dominant weights, which we now recall. More details are contained in \cite[Sec.~13.1]{Humphreys1972}.
Let $\ell$ be the rank of $\mathfrak g$, that is the dimension of $\mathfrak h_0$, so that we can label the simple roots form 1 to $\ell$, $\Sigma^+ = \{\gamma_1,\dots,\gamma_\ell\}$. Let $A = (A_{ij})$ be the \emph{Cartan matrix} associated to the Lie algebra $\mathfrak{g}_c$, which is defined as
\begin{equation}\label{eq:cartanMatrix}
A_{ij} = \frac{2 (\gamma_i,\gamma_j)}{(\gamma_i,\gamma_i)}.
\end{equation}
Then the \emph{fundamental dominant weights} $\varphi_1,\dots, \varphi_\ell$ are the elements of $\mathfrak h_{\mathbf R}$ defined by $\varphi_j = \sum_{i=1}^\ell (A^{-1})^{ij}\gamma_i$. In particular they form a basis of $\mathfrak h_{\mathbf R}^*$ suited for the computations we are going to perform. 

\begin{lem}\label{lem::vassumoffunddomweights}
Let $\psi \in \mathfrak h_{\mathbf R}^*$ and write $\psi = \sum_{j=1}^\ell w^j \varphi_j$ for some reals $w^1, \dots, w^\ell$.
Then one has $(\psi,\alpha) \geq 0$ for each positive root $\alpha$ if and only if all $w^i$'s are non-negative.
Moreover, denoted by $\Delta_+ \setminus \psi^\perp$ the subset of positive roots which are not orthogonal to $\psi$, one has $\Delta_+ \setminus \psi^\perp = \spann \left\{ \gamma_j \ |\ w^j \neq 0\right\} \cap \Delta_+$.
\end{lem}

The above lemma implies that the fundamental Weyl chamber $C$ is the (closed) convex cone spanned by the fundamental dominant weights $\varphi_1,\dots,\varphi_\ell$. Therefore we can write
\begin{equation}
\varphi = \sum_{i=1}^\ell v^i \varphi_i \quad \mbox{ for some } \quad v^1,\dots,v^\ell \geq 0,
\end{equation}
since we chose $\varphi\in C$.

We also recall an element of the root lattice that will come up frequently in our computation, which is denoted by $\delta$ \cite[Sec.~10.2]{Humphreys1972}. It is defined equivalently in terms of the roots or the fundamental dominant weights by
\begin{equation}\label{eq: deltavector}
\delta=\frac{1}{2}\sum_{\alpha\in\Delta_+}\alpha=\sum_{i=1}^{\ell}\varphi_i.
\end{equation}

\subsection{Canonical almost-complex structure}\label{sec::definitionJ}

In this section we recall a canonically defined homogeneous almost-complex structure on $M=G/V$, which turns out to be compatible with the Kirillov-Kostant-Souriau symplectic form $\omega$ \cite{DellaVedova2019,AlekseevskyPodesta2018}. An \emph{almost-complex structure} on $M$ is an element $J\in \text{End}(TM)$ such that $J^2=\text{id}$, and it is \emph{compatible} with $\omega$ if
\begin{equation}
\omega(JX,JY)=\omega(X,Y),\quad \omega(JX,X)>0,\quad\forall X,Y\in TM.
\end{equation}

Consider a root $\alpha \in \Delta$ and define $\lambda_\alpha = s_{\alpha}(\alpha,\varphi)\in \RR$, where $s_{\alpha}=1$ if $\alpha\in\Delta_+$ and $s_{\alpha}=-1$ otherwise. Note that, by the assumption $\varphi\in C$, it holds that $\lambda_\alpha \geq 0$, and $\lambda_\alpha=0$ when $\alpha$ is orthogonal to $\varphi$. Then, for each root $\alpha \in \Delta_+$ define the vectors
\begin{equation}\label{eq::uandvintermsofe}
u_\alpha = \frac{i^{(1-\varepsilon_\alpha)/2}}{\sqrt 2} (e_\alpha + e_{-\alpha}), \qquad
v_\alpha = \frac{i^{(3-\varepsilon_\alpha)/2}}{\sqrt 2} s_\alpha (e_\alpha - e_{-\alpha}),
\end{equation}
where $e_{\alpha}\in\mathfrak{g}^{\alpha}$ is an element chosen such as $[e_{\alpha},e_{-\alpha}]=h_{\alpha}$ and satisfying other additional properties \cite[Pg.~265]{GriffithsSchmid1969}. Observe that one has $u_\alpha = u_{-\alpha}$ and similarly $v_\alpha = v_{-\alpha}$. As a consequence of the choices of the $e_{\alpha}$'s one has the following statement.
\begin{lem}\label{lem::advonualphaandvalpha}
For all roots $\alpha, \beta \in \Delta$ one has
\begin{enumerate}
\item $B(u_\alpha,u_\beta) = B(v_\alpha,v_\beta) = (\delta_{\alpha,\beta} + \delta_{\alpha,-\beta})\varepsilon_\alpha$, \label{item::Bualphaubeta}
\item $B(u_\alpha,v_\beta)=0$, \label{item::Bualphavbeta}
\item $u_\alpha, v_\alpha \in \mathfrak g$, \label{item::alphavalphaing}
\item $[v,u_\alpha] = \lambda_\alpha v_\alpha$ and $[v,v_\alpha] = -\lambda_\alpha u_\alpha$, \label{item::[v,ualpha]}
\end{enumerate}
where $\delta_{a,b}=1$ if $a=b$ and $0$ otherwise.
\end{lem}

As a consequence of the above lemma, we have the $B$-orthogonal decomposition
\begin{equation}\label{eq:decompositiong}
\mathfrak g = \mathfrak h_0 \oplus \sum_{\alpha \in \Delta_+} \Span\{u_\alpha ,v_\alpha\}=\mathfrak h_0\oplus\sum_{\alpha \in \Delta_+\cap \varphi^\perp} \Span\{u_\alpha ,v_\alpha\}\oplus\sum_{\alpha \in \Delta_+\setminus \varphi^\perp} \Span\{u_\alpha ,v_\alpha\},
\end{equation}
where $\Delta_+ \cap \varphi^\perp$ denotes the subset of positive roots which are orthogonal to $\varphi$. Note that, as a consequence of point \ref{item::Bualphavbeta} of lemma \ref{lem::advonualphaandvalpha}, a root $\alpha$ is in $\Delta_+ \cap \varphi^\perp$ if $\lambda_\alpha=0$ and in its complement if $\lambda_\alpha>0$. Thus, item \ref{item::[v,ualpha]} of Lemma \ref{lem::advonualphaandvalpha} shows that the Lie algebra $\mathfrak v$ of the stabilizer of $v$ can be decomposed as 
\begin{equation}\label{eq:decompositionv}
\mathfrak{v}=\mathfrak h_0\oplus\sum_{\alpha \in \Delta_+\cap \varphi^\perp} \Span\{u_\alpha ,v_
\alpha\}.
\end{equation}
In particular, all roots belonging to $\Delta_+ \cap \varphi^\perp$ are compact, by compactness of $V$. On the other hand, putting together the decompositions \eqref{eq:decompositionv} and \eqref{eq:decompositiong}, one has 
\begin{equation}
\mathfrak{m}=\sum_{\alpha \in \Delta_+\setminus \varphi^\perp} \Span\{u_\alpha ,v_\alpha\}.
\end{equation}

Finally, let $H$ be the complex structure on $\mathfrak h_0^\perp = \Span_{\alpha \in \Delta_+} \{u_\alpha,v_\alpha\}$ defined by
\begin{equation}\label{eq::defH}
Hu_\alpha = \varepsilon_\alpha v_\alpha, \qquad Hv_\alpha = - \varepsilon_\alpha u_\alpha \qquad \mbox{for all } \alpha \in \Delta.
\end{equation}
Recall that $H:\mathfrak h_0^\perp\to \mathfrak h_0^\perp$ and $H^2=\text{id}_{\mathfrak h_0^\perp}$. We denote by $J\in\text{End}(T(G/V))$ the canonical homogeneous almost-complex structure on the orbit $G/V$ induced by $H$ on $\mathfrak m$. A manifold equipped with a symplectic form and a compatible almost-complex structure is called \emph{almost-K\"ahler}, hence $(G/V,\omega,J)$ turns out to be an almost-K\"ahler manifold.


\subsection{The condition $\rho = \lambda \omega$}\label{sec::spacialitycondition}

In this section we recall the necessary basic notions of symplectic geometry and the differential equation of interest. Note that, in our setting, the equation $\rho = \lambda \omega$ will be written as an algebraic equation, and we will deal only with that.

Let $(M,\omega)$ be a symplectic manifold, and let $J$ be a compatible almost-complex structure on it. We can define a closed two-form $\rho$ on $M$, the \emph{Chern-Ricci form} of $J$, in the following way. Let $\nabla$ be the \emph{Chern connection} on $M$, i.e., the unique affine connection on $M$ such that  $\nabla \omega = 0$, $\nabla J = 0$ and its torsion is exactly the Nijenhuis tensor of $J$. Its curvature $R$ is a two-form with values in $\text{End}(TM)$, and the form $\rho$ is defined by $\rho(X,Y) = \tr(JR(X,Y))$. 

As explained in the introduction, a common question to ask is if the equation $\rho = \lambda \omega$ is satisfied for some constant $\lambda\in\RR$. If this last equation is satisfied and $J$ is integrable (i.e., $M$ is a complex manifold), then $(M,\omega,J)$ is called \emph{K\"ahler-Einstein} manifold. On the other hand, a non-complex almost-K\"ahler manifold $(M,\omega,J)$ which satisfies $\rho = \lambda \omega$ is sometimes called \emph{Hermitian-Einstein}, \emph{special} \cite{DellaVedova2019}, or \emph{Chern-Einstein} \cite{AlekseevskyPodesta2018}. The nomenclature is not standard in this case. Below we are going to consider the condition $\rho = \lambda \omega$ on adjoint orbits \ref{sec:adjointOrbits} equipped with the almost-complex structure induced by $H$ \ref{sec::definitionJ}.

In our setting, the Chern-Ricci form $\rho$ of $J$ is determined by the two form $B(v',[\cdot,\cdot])$, with $ v'= 2 \sum_{\alpha \in \Delta_+ \setminus \varphi^\perp} [u_\alpha,v_\alpha]$ \cite[Sec.~4.2]{DellaVedova2019}.
By using the definition of $u_\alpha,v_\alpha$ \eqref{eq::uandvintermsofe}, one can write
\begin{equation}\label{eq::v'sumofroots}
v' = -2i \sum_{\alpha \in \Delta_+ \setminus \varphi^\perp} \varepsilon_\alpha h_\alpha.
\end{equation}
Hence, one can consider the element of the root lattice 
\begin{equation}\label{eq::definitionvarphi'}
\varphi' = - 2 \sum_{\alpha \in \Delta_+ \setminus \varphi^\perp} \varepsilon_\alpha \alpha \in \mathfrak h_{\mathbf R}^*
\end{equation}
satisfying $h_{\varphi'} = -iv'$. At this point, the condition involving the Chern-Ricci form $\rho = \lambda \omega$ turns out to be equivalent to $v' = \lambda v$ and $\varphi' = \lambda \varphi$.

We introduce the element of the root lattice 
\begin{equation}\label{eq:eta}
\eta = -2 \sum_{\alpha \in \Delta_+} \varepsilon_\alpha \alpha \in \mathfrak h_{\mathbf R}^*,
\end{equation}
that allows to write the condition in a more tractable way. This element can be thought as analogous to the $\delta$-vector \eqref{eq: deltavector} for real simple Lie algebras. Observe that $\eta$ depends on the semisimple Lie algebra $\mathfrak g$ and on the chosen set of positive roots. One can then write the condition $\rho = \lambda \omega$ as
\begin{equation}\label{eq::rho=lambdaomegaintermsoftau}
\eta - 2 \sum_{\alpha \in \Delta_+ \cap \varphi^\perp} \alpha = \lambda \varphi,
\end{equation}
where we used the fact that no non-compact roots can be orthogonal to $\varphi$. This is the equation we are going to study in section \ref{sec:specialCVD}, in terms of the combinatorics of Vogan diagrams. In particular, one has the following result, which is the building block to determine speciality of a given adjoint orbit. 
\begin{thm}\label{cor::specialvALL}
For any $\varphi \in \mathfrak h_{\mathbf R}^*$ and any real $\lambda \in\{-1,0,1\}$ the following are equivalent:
\begin{itemize}
\item $\varphi$ belongs to the dominant Weyl chamber $C$, the stabilizer of $v=ih_\varphi$ is compact, and one has $\varphi'=\lambda \varphi$;
\item there exists $S \subset \{ 1,\dots,\ell \}$ such that $i \in S$ whenever $\gamma_i$ is a non-compact simple root, and $\varphi_S = \eta - 2 \sum_{\alpha \in \Span\{ \gamma_i | i \in S^c\} \cap \Delta_+} \alpha$ satisfies $(\varphi_S,\gamma_i) = \lambda |(\varphi_S,\gamma_i)|$ for all $i \in S$. Moreover 
\begin{equation}
\varphi= \left\{ 
\begin{array}{ll}
\lambda \varphi_S & \mbox{if } \lambda = \pm 1 \\
\sum_{i \in S} v^i \varphi_i \mbox{ for some }v^i>0 & \mbox{if } \lambda = 0.
\end{array}
\right.
\end{equation}
\end{itemize}
\end{thm} 

What the above theorem says is that we can choose a Vogan diagram and algorithmically check the signs of $(\varphi_S,\gamma_i)$ for each simple non-compact root $\gamma_i$. If the signs are all the same, then the associated orbit admits special canonical almost-complex structure.

\subsection{Vogan diagrams}\label{sec::Vogandiagrams}

Vogan diagrams are combinatorial objects used to classify real semisimple Lie algebras \cite[Chapter VI]{Knapp1996}. In our setting, they will play an important role in studying the equation \eqref{eq::rho=lambdaomegaintermsoftau} on adjoint orbits of simple Lie groups, as explained in the previous sections.

A \emph{Vogan diagram with trivial automorphism} is a Dynkin diagram with some (including no one or all) painted vertices, where painted vertices correspond to simple non-compact roots. In the following, by Vogan diagram we will always intend a Vogan diagram with trivial automorphism. The reason why these diagrams are important in our setting is contained in the following lemma.

\begin{lem}\label{lem::correspVogandiagramwithv}
Let $G$ be a real semisimple Lie group with Lie algebra $\mathfrak g$, and let $\ell$ be the rank of $\mathfrak g$.
To any $v \in \mathfrak g$ with compact stabilizer, one can associate a Vogan diagram
and a vector $(v^1,\dots,v^\ell) \in \mathbf R^\ell$ with $v^i \geq 0$.
Moreover $v^i>0$ if the $i$-th node of the Vogan diagram is painted.
\end{lem}

Since a Cartan subalgebra $\mathfrak h_0$ containing $v$ and a Weyl chamber $C$ containing $\varphi$ cannot be chosen in a canonical way, one has that it is possible to associate different Vogan diagrams to the same element $v$. However, once the Vogan diagram is chosen and the simple roots are labelled, the vector $(v^1,\dots,v^\ell)$ is uniquely determined.

The correspondence established by Lemma \ref{lem::correspVogandiagramwithv} can be reversed. Indeed, given a connected Vogan diagram one can find a positive root system $\Delta_+$, and determine which roots in $\Delta_+$ are compact by using the formula in lemma \ref{cor::epsilonalphaformula}. At this point, the Weyl Chamber associated to $\Delta_+$ is determined, hence one can choose $\varphi \in C$, $\varphi = \sum_{i=1}^\ell v^i \varphi_i$, with $v^i \geq 0$.

The important fact is that, given a non-compact real simple Lie group $G$, all adjoint orbits $(G/V,\omega,J)$ such that $\rho=\lambda\omega$, for some constant $\lambda\in\RR$, can be listed in an algorithmic way (up to isomorphism and scaling). Indeed, by what we said above, this is equivalent to list (up to scaling) all special $\varphi$'s for all possible connected Vogan diagrams.
In \cite{DellaVedovaGatti2022} the authors classified all such vectors for classical Vogan diagrams up to rank $\ell=4$ and all the exceptional ones. This bound on the rank was mainly due to the fact that the algorithm to compute them becomes very slow with high ranks, since it requires to compute the full root system of the simple Lie algebra, together with its compact and non-compact roots. In the next section we are going to write the condition of special vector purely in terms of the indices of simple non-compact roots of a Vogan diagram. This makes the problem combinatorial and allows to scale the computation to much bigger Vogan diagrams.

\section{The condition $\rho=\lambda\omega$ through the combinatorics of Vogan diagrams}\label{sec:specialCVD}

In this section we are going to study the condition $\rho=\lambda\omega$ in terms of the combinatorics of Vogan diagrams. In particular, we will improve the results in \cite{DellaVedovaGatti2022}, providing a purely combinatorial condition for a special vector. Note that we will not consider adjoint orbits of exceptional type, since they were fully classified in \cite{DellaVedovaGatti2022}. The code that we used to test the formulae is available at the GitHub repository \href{https://github.com/alga-hopf/special-vogan-diagrams}{https://github.com/alga-hopf/special-vogan-diagrams}.

Consider the Vogan diagram of a non-compact real semisimple Lie algebra $\mathfrak{g}$ with $S=\lrbr{i_1,\ldots,i_m}$ the set of indices of simple non-compact roots. The vector
\begin{equation}
\varphi=\eta-2\sum_{\alpha\in\spann\lrbr{\gamma_i \vert i\in S^c}}\alpha
\end{equation}
can be written as
\begin{align}
\varphi =&\eta-2\sum_{\alpha\in\spann\lrbr{\gamma_i \vert i\in S^c}}\alpha\\
=& -2\sum_{\alpha\in\Delta_+^{nc}}\alpha-2\sum_{\alpha\in\Delta_+^c}\alpha +4\sum_{\alpha\in\Delta_+^c}-2\sum_{\alpha\in\spann\lrbr{\gamma_i \vert i\in S^c}}\alpha \label{eq:phisp1}\\
=& -4\delta+4\sum_{\alpha\in\Delta_+^c}\alpha-2\sum_{\alpha\in\spann\lrbr{\gamma_i \vert i\in S^c}}\alpha \label{eq:phisp4},
\end{align}
where in \eqref{eq:phisp1} we used the definition of $\eta$ \eqref{eq:eta}, in \eqref{eq:phisp4} we used the definition of $\delta$ \eqref{eq: deltavector} and $\Delta_+^{nc}$, $\Delta_+^{c}$ denote the sets of positive non-compact and compact roots respectively. To understand whether the vector $\varphi$ is special, we have to compute the signs of its $i$th coefficients, $i\in S$, by theorem \ref{cor::specialvALL}. As explained in section \ref{ssec:fundamentalDominantWeights}, the coefficients $\xi_i$ of $\varphi$ in the basis of the fundamental dominant weights $\lrbr{\omega_1,\ldots,\omega_{\ell}}$ are given by
\begin{equation}
\xi_i=(A\varphi)_i\quad i\in S,
\end{equation}
where $A$ is the Cartan matrix of the underlying complex Lie algebra \eqref{eq:cartanMatrix}. Hence
\begin{equation}
\tilde{\varphi}=-4A\delta+4A\sum_{\alpha\in\Delta_+^c}\alpha-2A\sum_{\alpha\in\spann\lrbr{\gamma_i \vert i\in S^c}}\alpha,
\end{equation}
where $\tilde{\varphi}$ denotes $\varphi$ in the basis of fundamental dominant weights. Since, by definition, $\delta_i=1$ in the basis of the fundamental dominant weights, we have the following expression for the $\xi_i$'s.
\begin{lem}\label{eq:xiExpression}
In the basis of the fundamental dominant weights, the coefficients $\xi_i$'s can be expressed as
\begin{equation}
\xi_i=-4+4\left(A\sum_{\alpha\in\Delta_+^c}\alpha\right)_i-2\left(A\sum_{\alpha\in\spann\lrbr{\gamma_i\vert i\in S^c}}\alpha\right)_i,
\end{equation}
where $A$ is the Cartan matrix of the underlying Dynkin diagram.
\end{lem}
Then, by theorem \ref{cor::specialvALL}, the considered Vogan diagram is special if all the $\xi_i$'s, $i\in S$, have the same sign. In order to check the signs of the $\xi_i$'s we need to compute explicitly the coefficients $\xi_i$, $i\in S$, which is the goal of the subsequent sections.

We can now start diving into the study of the coefficients $\xi_{i}$'s for Vogan diagrams of type $A_{\ell}$. 


\subsection{$A_{\ell}$ family}\label{sec:anCoefficients}
In this section we provide a closed formula for the coefficients $\xi_i$ defined above. The arguments used in this section will be used also for other families of diagrams, thus we provide the details here without repeating them in other cases, if not necessary.

Consider a diagram of type $A_{\ell}$ with $S=\lrbr{i_1,\ldots,i_m}$ the set of indices of simple non-compact roots, and let $\{\gamma_1,\ldots,\gamma_{\ell}\}$ denote the simple roots. Note that, by definition,
\begin{equation}\label{eq: coeffAn}
\left(\sum_{\alpha\in\spann\lrbr{\gamma_i \vert i\in S^c}}\alpha\right)_i=0\quad \forall i\in S.
\end{equation}
Combining equation \eqref{eq: coeffAn} together with the definition of the Cartan matrix $A$ for Lie algebras of type $A_{\ell}$ \cite[Sec.~11.4, Tab.~1]{Humphreys1972} we get for $i_k\in S$, $1\leq k\leq m$,
\begin{equation}
\left( A\sum_{\alpha\in\spann\lrbr{\gamma_i \vert i\in S^c}}\alpha\right)_{i_k}=-\dime(\hh_{i_{k}})-\dime(\hh_{i_{k+1}}),
\end{equation}
where $\hh_{i_k}$ is the Cartan subalgebra of the Lie subalgebra included between the nodes $i_{k-1}+1$ and $i_{k}-1$. This follows from the fact that the roots are unbroken strings \cite[Sec.~9.4]{Humphreys1972} and that the roots of the smaller subalgebra of type $A$ included between $i_{k-1}+1$ and $i_{k}-1$ having $\alpha_{i_k-1}\neq 0$ are exactly $\sum_{i=n_1}^{i_k-1}\gamma_i$, $i_{k-1}<n_1<i_k$. Their number corresponds to the dimension of the corresponding Cartan subalgebra $\dime(\hh_{i_{k}})$. A similar argument holds for $\dime(\hh_{i_{k+1}})$. Moreover, we can write
\begin{equation}\label{eq:compactRootsSplit}
\sum_{\alpha\in\Delta_+^c}\alpha=\sum_{\alpha\in\spann\lrbr{\gamma_i\vert i\in S^c}}\alpha+\sum_{\alpha\in\Delta_+^c\setminus\spann\lrbr{\gamma_i \vert i\in S^c}}\alpha,
\end{equation}
and we have
\begin{align}
\left(A\sum_{\alpha\in\Delta_+^c}\alpha\right)_{i_k}=& \left(A\sum_{\alpha\in\spann\lrbr{\gamma_i \vert i\in S^c}}\alpha\right)_{i_k}+\left(A\sum_{\alpha\in\Delta_+^c\setminus\spann\lrbr{\gamma_i \vert i\in S^c}}\alpha\right)_{i_k}\\
=&-\dime(\hh_{i_{k}})-\dime(\hh_{i_{k+1}})+\left(A\sum_{\alpha\in\Delta_+^c\setminus\spann\lrbr{\gamma_i \vert i\in S^c}}\alpha\right)_{i_k}\\
=& -\dime(\hh_{i_{k}})-\dime(\hh_{i_{k+1}})+\left(A\delta^c\right)_{i_k}, \label{eq:acompactroots}
\end{align}
where $\delta^c=\sum_{\alpha\in\Delta_+^c\setminus\spann\lrbr{\gamma_i\vert i\in S^c}}\alpha$.

Putting together equation \eqref{eq:acompactroots} and lemma \ref{eq:xiExpression} we get
\begin{align}\label{eq: phiCoeff}
\xi_{i_k}=&2\left(-2-\dime(\hh_{i_{k}})-\dime(\hh_{i_{k+1}})+2\left(A\delta^c\right)_{i_k}\right)\\
=&2\left(-2-(i_{k}-i_{k-1}-1)-(i_{k+1}-i_k-1)+2(-\delta^c_{i_k-1}+2\delta^c_{i_k}-\delta^c_{i_k+1})\right) \label{eq: equality3an}\\
=&2\left(i_{k-1}-i_{k+1}+2(-\delta^c_{i_k-1}+2\delta^c_{i_k}-\delta^c_{i_k+1})\right) \label{eq: equality4an},
\end{align}
where in equality \eqref{eq: equality3an} we used the fact that $\dime(\hh_{i_{l}})=i_l-i_{l-1}-1$, with the assumptions that $i_0=0$ and $i_{m+1}=\ell+1$, and the structure of $A$.

At this point, in order to determine the coefficient $\xi_{i_k}$, we need to make the term $-\delta^c_{i_k-1}+2\delta^c_{i_k}-\delta^c_{i_k+1}$, $i_k\in S$, explicit. This requires some knowledge about the roots systems of Lie algebras of type $A_{\ell}$. In particular, the positive roots of Lie algebras type $A_{\ell}$ are
\begin{equation}
\Delta_+=\left\{ \sum_{i=n_1}^{n_2}\gamma_i\ \vert\ 1\leq n_1\leq n_2\leq \ell\right\},
\end{equation}
while the positive compact roots are of the form
\begin{equation}
\Delta^c_+=\left\{ \sum_{i=n_1}^{n_2}\gamma_i\ \vert\ 1\leq n_1\leq n_2\leq \ell,\ \lvert S\cap\{n_1,\ldots,n_2\}\rvert\in 2\ZZ\right\}.
\end{equation}
Before stating the theorem about the coefficients $\xi_i$ we introduce the following quantities that will be used extensively through the paper:
\begin{align}
S_1(j):=&\sum_{k=1}^{\lceil\frac{m -j}{2} \rceil}i_{j+2k}-i_{j+2k-1},\quad S_2(j):=\sum_{k=1}^{\lfloor\frac{j}{2} \rfloor} i_{j-2k+1}-i_{j-2k} \\
T_1(j):=&\sum_{k=1}^{\lfloor\frac{j+1}{2} \rfloor-1} i_{j-2k}-i_{j-2k-1},\quad T_2(j):=\sum_{k=1}^{\lceil\frac{m -j-1}{2} \rceil}i_{j+2k+1}-i_{j+2k}\\
S_+(j):=&\sum_{k=j+1}^m 2(-1)^{k-j}i_k +i_{j+1}+(-1)^{m-j+1}(\ell+1)=S_1(j)-T_2(j) \label{eq: s+}\\
S_-(j):=&\sum_{k=1}^{j-1} 2(-1)^{k-j}i_k+i_{j-1}=S_2(j)-T_1(j). \label{eq: s-}
\end{align}
This notation is going to greatly simplify the formulae we are going to present.

\begin{thm}\label{thm: ancoeff}
Given a Vogan diagram of type $A_{\ell}$ with $S=\lrbr{i_1,\ldots,i_m}$ the set of indices of simple non-compact roots, for $i_j\in S$ one has
\begin{equation}\label{eq:xiAn}
\xi_{i_j}=2\left(i_{j-1}-i_{j+1}+2\left(S_+(j)-S_-(j)\right)\right).
\end{equation}
\end{thm}
\begin{proof}
First, assume that $i_0=0$ and $i_{m+1}=\ell+1$. We are going to prove that
\begin{equation}\label{eq: deltacThm}
-\delta_{i_j-1}^c+2\delta^c_{i_j}-\delta_{i_j+1}^c= S_1(j)+S_2(j) -T_1(j)-T_2(j).
\end{equation}
This, together with equality \eqref{eq: equality4an} and quantities \ref{eq: s+}, \eqref{eq: s-}, proves identity \eqref{eq:xiAn}.

Note that we can express
\begin{equation}
\delta^c_{i_j-1}=\delta^c_{i_j}-s_1+t_1,\quad \delta^c_{i_j+1}=\delta^c_{i_j}-s_2+t_2,
\end{equation}
where 
\begin{itemize}
\item $s_1$ is the number of compact roots $\sum_{i=1}^{n}\alpha_i\gamma_i$ in $\Delta_+^c\setminus\spann\lrbr{\gamma_i \vert i\in S^c}$ such that \\$\alpha_{i_j}-\alpha_{i_j-1}=1$,
\item $s_2$ is the number of compact roots in $\Delta_+^c\setminus\spann\lrbr{\gamma_i \vert i\in S^c}$ such that $\alpha_{i_j}-\alpha_{i_j+1}=1$,
\item $t_1$ is the number of compact roots in $\Delta_+^c\setminus\spann\lrbr{\gamma_i \vert i\in S^c}$ such that $\alpha_{i_j-1}-\alpha_{i_j}=1$,
\item $t_2$ is the number of compact roots in $\Delta_+^c\setminus\spann\lrbr{\gamma_i \vert i\in S^c}$ such that $\alpha_{i_j+1}-\alpha_{i_j}=1$.
\end{itemize}
Hence
\begin{equation}
-\delta_{i_j-1}^c+2\delta^c_{i_j}-\delta_{i_j+1}^c=s_1-t_1+s_2-t_2.
\end{equation}
So it suffices to determine the numbers $s_1,s_2,t_1,t_2$ for each $i_j$. We start with the coefficient $s_1$ and we need to count the compact roots in $\Delta_+^c\setminus\spann\lrbr{\gamma_i \vert i\in S^c}$ such that $\alpha_{i_j}=1$ and $\alpha_{i_j-1}=0$. These are of the form $\sum_{l=i_j}^{n_1}\gamma_l$, with $S\cap|\lrbr{i_j,\ldots,n_1}|\in 2\ZZ$. These roots are 
\begin{equation}
s_1=\sum_{k=1}^{\lceil\frac{m -j}{2} \rceil}i_{j+2k}-i_{j+2k-1}=S_1(j).
\end{equation}
A similar argument can be used to compute $s_2$, and we need to compute the the compact roots in $\Delta_+^c\setminus\spann\lrbr{\gamma_i \vert i\in S^c}$ such that $\alpha_{i_j}=1$ and $\alpha_{i_j+1}=0$, which are of the form $\sum_{l=n_1}^{i_j}\gamma_l$, with $S\cap|\lrbr{n_1,\ldots,i_j}|\in 2\ZZ$. 
These are
\begin{equation}
s_2=\sum_{k=1}^{\lfloor\frac{j}{2} \rfloor} i_{j-2k+1}-i_{j-2k}=S_2(j).
\end{equation}
For the coefficient $t_1$, the compact roots in $\Delta_+^c\setminus\spann\lrbr{\gamma_i \vert i\in S^c}$ such that $\alpha_{i_j}=0$ and $\alpha_{i_j-1}=1$ are of the form $\sum_{l=n_1}^{i_j-1}\gamma_l$, with $S\cap|\lrbr{n_1,\ldots,i_j-1}|\in 2\ZZ$. 
This reads
\begin{equation}
t_1=\sum_{k=1}^{\lfloor\frac{j+1}{2} \rfloor-1} i_{j-2k}-i_{j-2k-1}=T_1(j).
\end{equation}
Finally, for $t_2$ the compact roots in $\Delta_+^c\setminus\spann\lrbr{\gamma_i \vert i\in S^c}$ such that $\alpha_{i_j}=0$ and $\alpha_{i_j+1}=1$ are of the form $\sum_{l=i_j+1}^{n_2}\gamma_l$, with $S\cap|\lrbr{i_j+1,\ldots,n_1}|\in 2\ZZ$. So
\begin{equation}
t_2=\sum_{k=1}^{\lceil\frac{m -j-1}{2} \rceil}i_{j+2k+1}-i_{j+2k}=S_2(j).
\end{equation}
Thus,
\begin{equation}
s_1-t_1+s_2-t_2=S_1(j)-T_1(j)+S_2(j)-T_2(j).
\end{equation}
This proves equality \eqref{eq: deltacThm}
\end{proof}

Theorem \ref{thm: ancoeff} above gives a formula to compute the coefficients $\xi_i$ that depends only on the set $S$ and its combinatorics. In the next sections we are going to state similar results to the one for $A_{\ell}$. Even if the root systems are more intricate, the fundamental ideas are the ones developed in the current section.


\subsection{$B_{\ell}$ family}\label{sec:bnCoefficients}

In this section we are going to study the coefficients $\xi_i$ for Vogan diagrams of type $B_{\ell}$.  Consider a Vogan diagram of type $B_{\ell}$ with $S=\lrbr{i_1,\ldots,i_m}$ the set of indices of simple non-compact roots, and denote $\{\gamma_1,\ldots,\gamma_{\ell}\}$ the set of simple roots. As we did in section \ref{sec:anCoefficients}, taking into account the form of the Cartan matrix $A$ of type $B_{\ell}$ \cite[Sec.~11.4,Tab.~1]{Humphreys1972}, one has
\begin{equation}
\begin{split}
\Gamma_{i_k}=&\left( A\sum_{\alpha\in\spann\lrbr{\gamma_i \vert i\in S^c}}\alpha\right)_{i_k} \\
=&\left\{ \begin{array}{lcl}
-\dime(\hh_{i_{k}})-\dime(\hh_{i_{k+1}}) & \text{if} &k\neq m\\
-\dime(\hh_{i_{m}})-(2\dime(\hh_{i_{m+1}})-1)& \text{if} &k= m, i_k\neq \ell \\
-2\dime(\hh_{\ell}) & \text{if} & $k=m$, i_k= \ell \end{array}\right. \\
=&\left\{ \begin{array}{lcl}
i_{k-1}-i_{k+1}+2 & \text{if} &k\neq m\\
2+i_m+i_{m-1}-2n & \text{if} &k= m, i_k\neq \ell \\
-2n+2i_{m-1}+2 & \text{if} & k=m, i_k= \ell \end{array}\right.,
\end{split}
\end{equation}
where we assume $i_0=0$ and $i_{m+1}=\ell+1$. The case $k= m, i_k\neq \ell$ comes from the fact that the roots in $\spann\lrbr{\gamma_i \vert i\in S^c}$ that have non-zero $i_m+1$ component are the roots of the Lie algebra of type $B$ included between $i_m+1$ and $\ell$, which are exactly $2\dime(\hh_{i_{k+1}})-1$. The case $k=m, i_k= \ell$ comes from the structure of the Cartan matrix $A$ and the fact that the Lie algebra included between $i_{m-1}+1$ and $i_{m}-1$ is of type $A_{n}$.

By equation \eqref{eq:compactRootsSplit}, similarly as for the $A_{\ell}$ case, we have
\begin{equation}
\begin{split}
\left(A\sum_{\alpha\in\Delta_+^c}\alpha\right)_{i_k}=& \left(A\sum_{\alpha\in\spann\lrbr{\gamma_i \vert i\in S^c}}\alpha\right)_{i_k} + \left(A\sum_{\alpha\in\Delta_+^c\setminus\spann\lrbr{\gamma_i \vert i\in S^c}}\alpha\right)_{i_k}\\
=&\Gamma_{i_k}+\tau_{i_k},\\
\end{split}
\end{equation}
where $\delta^c=\sum_{\alpha\in\Delta_+^c\setminus\spann\lrbr{\gamma_i\vert i\in S^c}}\alpha$ and we put $\tau_{i_k}:=\left(A\delta^c\right)_{i_k}$. In particular,
\begin{equation}
\tau_k=\left\{\begin{array}{lcl}
-\delta^c_{i_k-1}+2\delta^c_{i_k}-\delta^c_{i_k+1} & \text{if} & i_k\neq \ell \\
-2\delta^c_{i_k-1}+2\delta^c_{i_k} & \text{if} & i_k=\ell
\end{array}\right.,
\end{equation}
by the structure of the Cartan matrix $A$.

Finally, as for the $A_{\ell}$ case,
\begin{equation}\label{eq: comformulabn}
\xi_{i_k}=-4-2\Gamma_{i_k}+4\left( A\sum_{\alpha\in\Delta_+^c}\alpha\right)_{i_k}=2\left(-2+\Gamma_{i_k}+2\tau_{i_k}\right).
\end{equation}
In order to write explicitly the coefficients $\tau_{i_k}$, we need to recall that the structure of the positive roots of type $B_{\ell}$,
\begin{equation}
\Delta_+=\left\{ \sum_{i=n_1}^{n_2}\gamma_i\ \vert\ 1\leq n_1\leq n_2\leq \ell\right\}\cup \left\{ \sum_{i=n_1}^{n_2}\gamma_i+\sum_{i=n_2+1}^{\ell}2 \gamma_i\ \vert\ 1\leq n_1\leq n_2\leq \ell-1\right\},
\end{equation}
and the positive compact roots
\begin{equation}
\begin{split}
\Delta^c_+=&\left\{ \sum_{i=n_1}^{n_2}\gamma_i\ \vert\ 1\leq n_1\leq n_2\leq \ell,\ \lvert S\cap\{n_1,\ldots,n_2\}\rvert \in 2\ZZ\right\}\cup \\ & \cup\left\{ \sum_{i=n_1}^{n_2}\gamma_i+\sum_{i=n_2+1}^{\ell}2 \gamma_i\ \vert\ 1\leq n_1\leq n_2\leq \ell-1,\ \lvert S\cap\{n_1,\ldots,n_2\}\rvert\in 2\ZZ\right\}.
\end{split}
\end{equation}

In the following theorem we compute the coefficients $\xi_i$.
\begin{thm}
Given a Vogan diagram of type $B_{\ell}$ with $S=\lrbr{i_1,\ldots,i_m}$ the set of indices of simple non-compact roots, for $i_j\in S$ one has
\begin{itemize}
\item If $j\neq m$
\begin{equation}\label{eq:bCoeff1}
\xi_{i_j}=2(-i_{j-1}+4i_j-3i_{j+1}+4S_+(j)+2(-1)^{j+m}).
\end{equation}
\item If $j= m$ and $i_m\neq \ell$
\begin{equation}\label{eq:bCoeff2}
\xi_{i_m}= 2(3i_m-i_{m-1}-2\ell-2).
\end{equation}
\item If $i_j=\ell$
\begin{equation}\label{eq:bCoeff3}
\xi_{\ell}= 2(2\ell-2i_{m-1}).
\end{equation}
\end{itemize}
\end{thm}

\begin{proof}
Assume that $i_0=0$ and $i_{m+1}=\ell+1$. 
\begin{itemize}
\item Case $j\neq m$. We are going to prove that 
\begin{equation}\label{eq: deltac_Bn1}
\tau_{i_j}= -i_{j-1}+2i_j-i_{j+1}+  2S_1(j)-2T_2(j)+(-1)^{j+m}.
\end{equation}
This, together with equality \eqref{eq: comformulabn} proves equality \eqref{eq:bCoeff1}. In this case, $\tau_{i_j}$ can be written as
\begin{equation}
\tau_{i_j}=-\delta^c_{i_k-1}+2\delta^c_{i_k}-\delta^c_{i_k+1}=s_1-t_1+s_2-t_2,
\end{equation}
as for the $A_{\ell}$ case. We start with the computation of $s_1$. In order to count the compact roots in $\Delta_+^c\setminus\spann\lrbr{\gamma_i \vert i\in S^c}$ such that $\alpha_{i_j}-\alpha_{i_j-1}=1$, one can first consider the roots such that $\alpha_{i_j}=1$ and $\alpha_{i_j-1}=0$, which are of two types. The roots $\sum_{l=i_j}^{n_2}\gamma_l$, with $|S\cap\lrbr{i_j,\ldots,n_2}|\in 2\ZZ$, and $\sum_{l=i_j}^{n_1}\gamma_l+\sum_{l=n_1+1}^{\ell}2\gamma_l$, with $\lvert S\cap\{i_j,\ldots,n_1\}\rvert\in 2\ZZ$, that are exactly
\begin{equation}
\sum_{k=1}^{\lceil\frac{m -j}{2} \rceil}i_{j+2k}-i_{j+2k-1}+\sum_{k=1}^{\lceil\frac{m -j}{2} \rceil}i_{j+2k}-i_{j+2k-1}-\delta_{\ell+1,i_{j+2\lceil\frac{m -j}{2} \rceil}}=2S_1(j)-\delta_{\ell+1,i_{j+2\lceil\frac{m -j}{2} \rceil}}.
\end{equation}
Finally, we have to consider all the roots for which $\alpha_{i_j}=2$ and $\alpha_{i_j-1}=1$, that are of the form $\sum_{l=n_1}^{i_{j-1}}\gamma_l+\sum_{l=i_j}^{\ell}2\gamma_l$, with $\lvert S\cap\{n_1,\ldots,i_{j-1}\}\rvert\in 2\ZZ$. These are
\begin{equation}
\sum_{k=1}^{\lfloor\frac{j+1}{2} \rfloor-1} i_{j-2k}-i_{j-2k-1}+(i_j-i_{j-1}-1)=T_1(j)+i_j-i_{j-1}-1.
\end{equation}
So
\begin{equation}
s_1=2S_1(j)+T_1(j)+i_j-i_{j-1}-1-\delta_{\ell+1,i_{j+2\lceil\frac{m -j}{2} \rceil}}.
\end{equation}
For the coefficients $s_2$ and $t_1$ we can follow the same steps made for the $A_{\ell}$ case, so
\begin{align}
s_2=&\sum_{k=1}^{\lfloor\frac{j}{2} \rfloor} i_{j-2k+1}-i_{j-2k}=S_2(j),\\
t_1=&\sum_{k=1}^{\lfloor\frac{j+1}{2} \rfloor-1} i_{j-2k}-i_{j-2k-1}=T_1(j).
\end{align}
Finally, for $t_2$ one has that the roots in $\Delta_+^c\setminus\spann\lrbr{\gamma_i \vert i\in S^c}$ such that $\alpha_{i_j+1}-\alpha_{i_j}=1$ are the following. First we have the roots such that $\alpha_{i_j+1}=1$ and $\alpha_{i_j}=0$, that are of two types: the roots $\sum_{l=i_j+1}^{n_1}\gamma_l$, with $|S\cap\lrbr{i_j+1,\ldots,n_1}|\in 2\ZZ$, and $\sum_{l=i_j+1}^{n_1}\gamma_l+\sum_{l=n_1+1}^{\ell}2\gamma_l$, with $\lvert S\cap\{i_j+1,\ldots,n_1\}\rvert\in 2\ZZ$. These are
\begin{multline}
\sum_{k=1}^{\lceil\frac{m -j-1}{2} \rceil}i_{j+2k+1}-i_{j+2k}+\sum_{k=1}^{\lceil\frac{m -j-1}{2} \rceil}i_{j+2k+1}-i_{j+2k}+(i_{j+1}-i_j-1)+ \\-\delta_{\ell+1,i_{j+2\lceil\frac{m -(j+1)}{2} \rceil+1}}
=2T_2(j)+i_{j+1}-i_j-1-\delta_{\ell+1,i_{j+2\lceil\frac{m -(j+1)}{2} \rceil+1}}.
\end{multline}
Then there are the compact roots in $\Delta_+^c\setminus\spann\lrbr{\gamma_i \vert i\in S^c}$ such that $\alpha_{i_j+1}=2$ and $\alpha_{i_j}=1$, that are of the form $\sum_{l=n_1}^{i_j}\gamma_l+\sum_{l=i_j+1}^{\ell}2\gamma_l$, with $\lvert S\cap\{n_1,\ldots,i_j\}\rvert\in 2\ZZ$. These are
\begin{equation}
\sum_{k=1}^{\lfloor\frac{j}{2} \rfloor} i_{j-2k+1}-i_{j-2k}=S_2(j).
\end{equation}
Thus,
\begin{equation}
t_2=2T_2(j)+S_2(j)+i_{j+1}-i_j-1-\delta_{\ell+1,i_{2\lceil\frac{m -(j+1)}{2} \rceil+1}}.
\end{equation}
Putting everything together we get
\begin{equation}
s_1+s_2-t_1-t_2=  -i_{j-1}+2i_j-i_{j+1}+2S_1+(j)-2T_2(j)+(-1)^{j+m},
\end{equation}
where we used the fact that $-\delta_{\ell+1,i_{j+2\lceil\frac{m -j}{2} \rceil}}+\delta_{\ell+1,i_{j+2\lceil\frac{m -(j+1)}{2} \rceil+1}}=(-1)^{j+m}$.
This proves equality \eqref{eq: deltac_Bn1}.

\item Case $j=m$, $i_m\neq \ell$. We will prove that
\begin{equation}\label{eq: deltac_Bn2}
\tau_{i_m}= i_j-i_{j-1}-1,
\end{equation}
that, together with equality \eqref{eq: comformulabn}, proves equality \eqref{eq:bCoeff2}. 
Again, we have
\begin{equation}
\tau_{i_m}=-\delta^c_{i_m-1}+2\delta^c_{i_m}-\delta^c_{i_m+1}=s_1-t_1+s_2-t_2,
\end{equation}
so we only need to compute the coefficients $s_1,t_1,s_2,t_2$. We start with the coefficient $s_1$. Notice that the compact roots in $\Delta_+^c\setminus\spann\lrbr{\gamma_i \vert i\in S^c}$ such that $\alpha_{i_m}-\alpha_{i_m-1}=1$, are only the roots such that $\alpha_{i_m}=2$ and $\alpha_{i_m-1}=1$, that are of type $\sum_{l=n_1}^{i_{m-1}}\gamma_l+\sum_{l=i_m}^{\ell}2\gamma_l$, with $S\cap\lvert\{n_1,\ldots,i_{m-1}\}\rvert\in 2\ZZ$. So
\begin{equation}
s_1=\sum_{k=1}^{\lfloor\frac{m+1}{2} \rfloor-1} i_{m-2k}-i_{m-2k-1}+(i_m-i_{m-1}-1)=T_1(m)+i_m-i_{m-1}-1.
\end{equation}
As above, for $s_2$ and $t_1$ we have the equalities
\begin{equation}
s_2=\sum_{k=1}^{\lfloor\frac{j}{2} \rfloor} i_{m-2k+1}-i_{m-2k}=S_2(m),\quad t_1=\sum_{k=1}^{\lfloor\frac{m+1}{2} \rfloor-1} i_{m-2k}-i_{m-2k-1}=T_1(m).
\end{equation}
Finally, for $t_2$ one has that the roots in $\Delta_+^c\setminus\spann\lrbr{\gamma_i \vert i\in S^c}$ such that $\alpha_{i_m+1}-\alpha_{i_m}=1$ are the ones for which $\alpha_{i_m+1}=2$ and $\alpha_{i_m}=1$. These are $\sum_{l=n_1}^{i_m}\gamma_l+\sum_{l=i_m+1}^{\ell}2\gamma_l$, with $\lvert\{n_1,\ldots,i_m\}\cap S\rvert\in 2\ZZ$, so 
\begin{equation}
t_2=\sum_{k=1}^{\lfloor\frac{m}{2} \rfloor} i_{m-2k+1}-i_{m-2k}=S_2(m).
\end{equation}

Summing everything up we get
\begin{equation}
s_1+s_2-t_1-t_2= i_m-i_{m-1}-1,
\end{equation}
which proves equality \eqref{eq: deltac_Bn2}.

\item Case $j=m$, $i_m=\ell$. We will prove equality
\begin{equation}\label{eq: deltac_Bn3}
\tau_{\ell}= 2\left(\ell-i_{m-1}-1\right)
\end{equation}
that, together with equality \eqref{eq: comformulabn}, proves \eqref{eq:bCoeff3}. Remember that
\begin{equation}
\tau_{\ell}=-2\delta^c_{\ell-1}+2\delta^c_{\ell}=s_1-t_1.
\end{equation}
For $s_1$, the compact roots in $\Delta_+^c\setminus\spann\lrbr{\gamma_i \vert i\in S^c}$ such that $\alpha_{\ell}-\alpha_{\ell-1}=1$ are the ones for which $\alpha_{\ell}=2$ and $\alpha_{\ell-1}=1$, that are of type $\sum_{l=n_1}^{\ell-1}\gamma_l+2\gamma_{\ell}$, with $\lvert S\cap\{n_1,\ldots,\ell-1\}\rvert\in 2\ZZ$. These are
\begin{equation}
s_1=\sum_{k=1}^{\lfloor\frac{j+1}{2} \rfloor-1} i_{j-2k}-i_{j-2k-1}+(\ell-i_{m-1}-1)=T_1(j)+\ell-i_{m-1}-1.
\end{equation}
The coefficient $t_1$ is equal to $T_1(m)$, as in the previous cases, hence
\begin{equation}
s_1-t_1= \ell-i_{m-1}-1,
\end{equation}
which proves equality \eqref{eq: deltac_Bn3} and concludes the proof of the theorem.
\end{itemize}
\end{proof}

The above theorem allows us to compute the coefficients of Vogan diagrams of type $B_{\ell}$ and check their signs. The next two sections are dedicated to Vogan diagrams of type $C_{\ell}$ and $D_{\ell}$. The proofs in these cases are going to be significantly more intricate than the ones we have seen so far. Thus we will try to keep them short and refer to the previous proofs in case of similar computations.


\subsection{$C_{\ell}$ family}\label{sec:cnCoefficients}
Consider a Vogan diagram of type $C_{\ell}$, with $S=\{i_1,\ldots,i_m\}$ the set of indices of simple non-compact roots and denote by $\{\gamma_1,\ldots,\gamma_{\ell}\}$ the set of simple roots. As for the other cases, using the structure of the Cartan matrix $A$ for Lie algebras of type $C_{\ell}$ \cite[Sec.~11.4, Tab.~1]{Humphreys1972}, one can write
\begin{equation}\label{eq:gammaC}
\begin{split}
\Gamma_{i_k}=&\left( A\sum_{\alpha\in\spann\lrbr{\gamma_i \vert i\in S^c}}\alpha\right)_{i_k}\\
&=\left\{ \begin{array}{lcl}-\dime(\hh_{i_{k}})-\dime(\hh_{i_{k+1}}) & \text{if} &k\neq m,\ i_k\neq \ell-1\\
-\dime(\hh_{i_{m}})-2\dime(\hh_{i_{m+1}})& \text{if} &k= m, i_k\neq \ell-1 \\
-\dime(\hh_{\ell-1})-2 & \text{if} & k=m, i_k= \ell-1 \\
-\dime(\hh_{\ell-1}) & \text{if} & k\neq m, i_k= \ell-1 \\
\end{array}\right.\\
=& \left\{ \begin{array}{lcl}
i_{k-1}-i_{k+1}+2 & \text{if} &k\neq m,\ i_k\neq \ell-1\\
i_m+i_{m-1}+1-2\ell & \text{if} &k= m, i_k\neq \ell-1 \\
-\ell+i_{m-1} & \text{if} & k=m,i_k= \ell-1 \\
-\ell+i_{k-1}+2 & \text{if} & k\neq m,i_k= \ell-1 
\end{array}\right.
\end{split}
\end{equation}
where, for $k= m, i_k\neq \ell-1$, the second summand comes from the fact that the roots in $\spann\lrbr{\gamma_i \vert i\in S^c}$ that have non-zero $i_m+1$ component are the roots of the Lie algebra of type $C$ included between $i_m+1$ and $\ell$, which are exactly $2\dime(\hh_{i_{m+1}})$. For the last two cases, the second summand is equal to $0$ if $i_m\neq \ell-1$, and equal to $-2$ if $i_m= \ell-1$, as a consequence of the fact that $A_{\ell-1,\ell}=-2$, and there is only one root between $\ell-1$ and $\ell$.

Similarly for the $A_{\ell}$ and $B_{\ell}$ cases, one has
\begin{equation}\label{eq: formulacncoefficients}
\xi_{i_k}=-4-2\Gamma_{i_k}+4\left( A\sum_{\alpha\in\Delta_+^c}\alpha\right)_{i_k}=2\left(-2+\Gamma_{i_k}+2\tau_{i_k}\right),
\end{equation}
where $\delta^c=\sum_{\alpha\in\Delta_+^c\setminus\spann\lrbr{\gamma_i\vert i\in S^c}}\alpha$ and we put $\tau_{i_k}:=\left(A\delta^c\right)_{i_k}$. Note that
\begin{equation}
\tau_{i_k}=\left\{\begin{array}{lcl}
-\delta^c_{i_k-1}+2\delta^c_{i_k}-\delta^c_{i_k+1} & \text{if} & i_k\neq \ell-1 \\
-\delta^c_{\ell-1}+2\delta^c_{\ell} & \text{if} & i_k=\ell\\
-\delta^c_{\ell-2}+2\delta^c_{\ell-1}-2\delta^c_{\ell} & \text{if} & i_k=\ell-1\\
\end{array}\right.,
\end{equation}
again as a consequence of the structure of $A$.

The coefficients $\tau_{i_k}$ depend on the structure of the roots of type $C_{\ell}$, which we recall below. The positive roots of type $C_{\ell}$ are 
\begin{equation}
\begin{split}
\Delta_+=&\left\{ \sum_{i=n_1}^{n_2}\gamma_i\ \vert\ 1\leq n_1\leq n_2\leq \ell\right\}\cup \cup \left\{\sum_{i=n_1}^{\ell-1}2 \gamma_i+\gamma_{\ell}\ \vert\ 1\leq n_1\leq \ell-1\right\}\cup \\ & \left\{ \sum_{i=n_1}^{n_2}\gamma_i+\sum_{i=n_2+1}^{\ell-1}2 \gamma_i+\gamma_{\ell}\ \vert\ 1\leq n_1\leq n_2\leq \ell-2\right\},
\end{split}
\end{equation}
and the positive compact roots are
\begin{equation}
\begin{split}\label{eq:rootscn}
\Delta^c_+=&\left\{ \sum_{i=n_1}^{n_2}\gamma_i\ \vert\ 1\leq n_1\leq n_2\leq \ell,\ \lvert S\cap\{n_1,\ldots,n_2\}\rvert\in 2\ZZ\right\}\cup \\ 
& \cup\left\{ \sum_{i=n_1}^{n_2}\gamma_i+\sum_{i=n_2+1}^{\ell-1}2 \gamma_i+\gamma_{\ell}\ \vert\ 1\leq  n_1\leq n_2\leq \ell-2,\ \lvert S\cap \{n_1,\ldots,n_2,\ell\}\rvert\in 2\ZZ\right\}\cup\\
& \cup\left\{ \sum_{i=n_1}^{\ell-1}2 \gamma_i+\gamma_{\ell}\ \vert\ 1\leq n_1\leq n-1,\ \ell\notin S\right\}
\end{split}
\end{equation}

The following theorem shows the structure of the coefficients $\xi_i$.
\begin{thm}\label{thm: xiCnall}
Given a Vogan diagram of type $C_{\ell}$ with $S=\lrbr{i_1,\ldots,i_m}$ the set of indices of simple non-compact roots, for $i_j\in S$ one has
\begin{itemize}

\item If $i_j\neq \ell-1$

\begin{itemize}

\item If $j\neq m$ and $\ell\notin S$
\begin{equation}\label{eq: xi_Cn1}
\xi_{i_j}=2\left(-i_{j-1}+4i_j-3i_{j+1} +4S_+(j)+4(-1)^{j+m}\right).
\end{equation}

\item If $j\neq m$ and $\ell\in S$
\begin{equation}\label{eq: xi_Cn2}
\xi_{i_j}= 2\left(3i_{j-1}-4i_j+i_{j+1}+4S_-(j)\right).
\end{equation}

\item If $j=m$ and $i_j=\ell$
\begin{equation}\label{eq: xi_Cn3}
\xi_{\ell}= 2\left(\ell-i_{m-1}-3+2S_-(j) \right).
\end{equation}

\item If $j=m$ and $i_j\neq \ell$
\begin{equation}\label{eq: xi_Cn4}
\xi_{i_m}= 2\left(3i_m-i_{m-1}-2\ell+1\right).
\end{equation}
\end{itemize}

\item If $i_j=\ell-1$

\begin{itemize}

\item If $\ell\in S$
\begin{equation}\label{eq: xi_Cn5}
\xi_{\ell-1}= 2(\ell+3i_{j-1}+6+ 4S_-(j)).
\end{equation}

\item If $\ell\notin S$
\begin{equation}\label{eq: xi_Cn6}
\xi_{\ell-1}= 2(\ell-i_{j-1}).
\end{equation}

\end{itemize}
\end{itemize}
\end{thm}

\begin{proof}
Assume that $i_0=0$ and $i_{m+1}=\ell+1$. 
\begin{itemize}
\item Case $i_j\neq \ell-1$. 

\begin{itemize}
\item Case $j\neq m$, $\ell\notin S$. We will prove that 
\begin{equation}\label{eq: deltac_Cn1}
\tau_{i_j}=-i_{j-1}+2i_j-i_{j+1} +2(S_1(j)- T_2(j))+2(-1)^{j+m},
\end{equation}
that, together with equalities \eqref{eq:gammaC} and \eqref{eq: formulacncoefficients} proves \eqref{eq: xi_Cn1}. We need to compute
\begin{equation}
\tau_{i_j}=-\delta_{i_j-1}^c+2\delta^c_{i_j}-\delta_{i_j+1}^c=s_1-t_1+s_2-t_2.
\end{equation}
We start with the coefficient $s_1$. The compact roots in $\Delta^c_+\setminus\spann\lrbr{\gamma_i \vert i\in S^c}$ such that $\alpha_{i_j}-\alpha_{i_j-1}=1,2$ are the ones for which $\alpha_{i_j}=1$ and $\alpha_{i_j-1}=0$, or $\alpha_{i_j}=2$ and $\alpha_{i_j-1}=1$, or $\alpha_{i_j}=2$ and $\alpha_{i_j-1}=0$. Following similar arguments to the ones used in the proof of case $B_{\ell}$, and looking and the structure of the roots \eqref{eq:rootscn}, one has
\begin{equation}
\begin{split}
s_1=& 2\sum_{k=1}^{\lceil\frac{m -j}{2} \rceil}i_{j+2k}-i_{j+2k-1}+\sum_{k=1}^{\lfloor\frac{j+1}{2} \rfloor-1} i_{j-2k}+i_{j-2k-1}+(i_j-i_{j-1}-1)+2+ \\&-2\delta_{\ell+1,i_{j+2\lceil\frac{m -j}{2} \rceil}} \\
=& 2S_1(j)+T_1(j)+i_j-i_{j-1}+1-2\delta_{\ell+1,i_{j+2\lceil\frac{m -j}{2} \rceil}},
\end{split}
\end{equation}
where the term $\delta_{\ell+1,i_{j+2\lceil\frac{m -j}{2} \rceil}}$ comes from the fact that $\alpha_{\ell}$ is fixed and equal to $1$ in the roots $\sum_{l=i_j}^{n_1}\gamma_l+\sum_{l=n_1+1}^{\ell-1}2 \gamma_l+\gamma_{\ell}$ and $\sum_{l=n_1}^{i_j-1}\gamma_l+\sum_{l=i_j}^{\ell-1}2 \gamma_i+\gamma_{\ell}$.

As above, the coefficients $s_2$ and $t_1$ are precisely $s_2=S_2(j)$ and $t_1=T_1(j)$.

Similarly to $s_1$, for $t_2$ one has that the roots in $\Delta^c_+\setminus\spann\lrbr{\gamma_i \vert i\in S^c}$ such that $\alpha_{i_j+1}-\alpha_{i_j}=1,2$ are the ones for which $\alpha_{i_j+1}=1$ and $\alpha_{i_j}=0$, or $\alpha_{i_j+1}=2$ and $\alpha_{i_j}=1$, or $\alpha_{i_j+1}=2$ and $\alpha_{i_j}=0$, hence
\begin{equation}
\begin{split}
t_2=& 2\sum_{k=1}^{\lceil\frac{m -j-1}{2} \rceil}i_{j+2k+1}-i_{j+2k}+(i_{j+1}-i_j-1)+\sum_{k=1}^{\lfloor\frac{j}{2} \rfloor} i_{j-2k+1}-i_{j-2k}+2+\\&-2\delta_{\ell+1,i_{j+2\lceil\frac{m -(j+1)}{2} \rceil+1}} \\
=& i_{j+1}-i_j+1+2T_2(j)+S_1(j)-2\delta_{\ell+1,i_{j+2\lceil\frac{m -(j+1)}{2} \rceil+1}}
\end{split}
\end{equation}
Finally
\begin{equation}
s_1+s_2-t_1-t_2=-i_{j-1}+2i_j-i_{j+1}+2(S_1(j)- T_2(j))+2(-1)^{j+m}.
\end{equation}
This proves equality \eqref{eq: deltac_Cn1}.

\item Case $j\neq m$, $\ell\in S$. We will prove that
\begin{equation}\label{eq: deltac_Cn2}
\tau_{i_j}= i_{j-1}-2i_j+i_{j+1}+2(S_2(j) -T_1(j)).
\end{equation}
This, with equalities \eqref{eq:gammaC} and \eqref{eq: formulacncoefficients} proves \eqref{eq: xi_Cn2}. Recall that
\begin{equation}
\tau_{i_j}=-\delta_{i_j-1}^c+2\delta^c_{i_j}-\delta_{i_j+1}^c=s_1-t_1+s_2-t_2.
\end{equation}
The computations are similar as the previous case, except that the fact that $\gamma_{\ell}\in S$ changes the compactness of the roots with respect to the above point. Hence we have
\begin{equation}
\begin{split}
s_1=& \sum_{k=1}^{\lceil\frac{m -j}{2} \rceil}i_{j+2k}-i_{j+2k-1}+
\sum_{k=1}^{\lceil\frac{m -j-1}{2} \rceil}i_{j+2k+1}-i_{j+2k}+ (i_{j+1}-i_j)+\\
& + \sum_{k=1}^{\lfloor\frac{j}{2} \rfloor} i_{j-2k+1}+i_{j-2k}-2\delta_{\ell+1,j+2\lceil\frac{m -j}{2} \rceil} \\
=& i_{j+1}-i_j+S_1(j)+T_2(j)+S_2(j)-2\delta_{\ell+1,j+2\lceil\frac{m -j}{2} \rceil}\\
t_2=&\sum_{k=1}^{\lceil\frac{m -j-1}{2} \rceil}i_{j+2k+1}-i_{j+2k}+\sum_{k=1}^{\lceil\frac{m -j}{2} \rceil}i_{j+2k}-i_{j+2k-1}+ \\
&\sum_{k=1}^{\lfloor\frac{j+1}{2} \rfloor-1} i_{j-2k}-i_{j-2k-1}+(i_{j}-i_{j-1})-2\delta_{\ell+1,i_{j+2\lceil\frac{m -(j+1)}{2} \rceil+1}} \\
=& i_{j}-i_{j-1}+T_2(j)+S_1(j)+T_1(j)-2\delta_{\ell+1,i_{j+2\lceil\frac{m -(j+1)}{2} \rceil+1}},
\end{split}
\end{equation}
and, again $s_2=S_2(j)$, $t_1=T_1(j)$. Finally
\begin{equation}
s_1+s_2-t_1-t_2=i_{j-1}-2i_j+i_{j+1}+2(S_2(j) -T_1(j))+2(-1)^{j+m}.
\end{equation}
This proves equality \eqref{eq: deltac_Cn2}.

\item Case $j=m$ and $i_j=\ell$. We will prove that
\begin{equation}\label{eq: deltac_Cn3}
\tau_{\ell}= S_2(m)-T_1(m).
\end{equation}
This, with equalities \eqref{eq:gammaC} and \eqref{eq: formulacncoefficients} proves \eqref{eq: xi_Cn3}. In this case, $\tau_{\ell}=-\delta^c_{\ell-1}+2\delta^c_{\ell}$. Note that the compact roots in $\Delta^c_+\setminus\spann\lrbr{\gamma_i \vert i\in S^c}$ such that $\alpha_{\ell}\neq 0$ are
\begin{align}
& \sum_{k=n_1}^{\ell} \gamma_k,\quad \lvert S\cap\lrbr{n_1,\ldots,\ell}\rvert\in 2\ZZ\\
& \sum_{k=n_1}^{n_2} \gamma_k+\sum_{k=n_2+1}^{\ell-1} 2\gamma_k+\gamma_{\ell},\quad \lvert S\cap\lrbr{n_1,\ldots,n_2}\rvert\in 1+2\ZZ,
\end{align}
with coefficient $\alpha_{\ell}=1$. On the other hand, the compact roots in $\Delta^c_+\setminus\spann\lrbr{\gamma_i \vert i\in S^c}$ such that $\alpha_{\ell-1}\neq 0$ are
\begin{align}
& \sum_{k=n_1}^{\ell-1} \gamma_k,\quad \lvert S\cap\lrbr{n_1,\ldots,\ell-1}\rvert\in 2\ZZ\\
& \sum_{k=n_1}^{\ell-1} \gamma_k+\gamma_{\ell},\quad \lvert S\cap \lrbr{n_1,\ldots,\ell-1}\rvert\in 1+2\ZZ\\
& \sum_{k=n_1}^{n_2} \gamma_k+\sum_{k=n_2+1}^{\ell-1} 2\gamma_k+\gamma_{\ell},\quad \lvert S\cap\lrbr{n_1,\ldots,n_2}\rvert\in 1+2\ZZ,
\end{align}
with coefficients $\alpha_{\ell-1}=1,1,2$ respectively. Not all the roots contribute to the summation. In particular, the ones that do are
\begin{itemize}
\item $\sum_{k=n_1}^{\ell-1} \gamma_k,\ \lvert S\cap \lrbr{n_1,\ldots,\ell-1}\rvert\in 2\ZZ$, which contribute by $-\sum_{k=1}^{\lfloor\frac{j}{2} \rfloor} i_{j-2k+1}+i_{j-2k}$.
\item $\sum_{k=n_1}^{\ell-1} \gamma_k+\gamma_{\ell},\quad \lvert S\cap \lrbr{n_1,\ldots,\ell-1}\rvert\in 1+2\ZZ$, which contribute by $\sum_{k=1}^{\lfloor\frac{j+1}{2} \rfloor-1} i_{j-2k}-i_{j-2k-1}$.
\end{itemize}
By summing all the contributes we get
\begin{equation}
\tau_{\ell}= -\sum_{k=1}^{\lfloor\frac{j+1}{2} \rfloor-1} i_{j-2k}-i_{j-2k-1}+ \sum_{k=1}^{\lfloor\frac{j}{2} \rfloor} i_{j-2k+1}+i_{j-2k}=-T_1(j)+S_2(j).
\end{equation}
This proves equality \eqref{eq: deltac_Cn3}.

\item Case $j=m$ and $i_j\neq \ell$. We are going to prove that
\begin{equation}\label{eq: deltac_Cn4}
\tau_{i_m}= i_m-i_{m-1}+1.
\end{equation}
This, with equalities \eqref{eq:gammaC} and \eqref{eq: formulacncoefficients} proves \eqref{eq: xi_Cn4}. This case is similar to the case $j\neq m$, $\ell\notin S$, except that we do not have to take into account the roots with $\alpha_{m+1}\neq 0,\ldots,\alpha_{\ell}\neq 0$, since these are not included in $\Delta^c_+\setminus\spann\lrbr{\gamma_i \vert i\in S^c}$. Thus
\begin{equation}
\begin{split}
s_1=&\sum_{k=1}^{\lfloor\frac{j+1}{2} \rfloor-1} i_{m-2k}-i_{m-2k-1}+(i_m-i_{m-1}-1)+2=T_1(m)+i_m-i_{m-1}+1. \\
t_2=&\sum_{k=1}^{\lfloor\frac{j}{2} \rfloor} i_{m-2k+1}-i_{m-2k}=S_2(m), 
\end{split}
\end{equation}
and $s_2=S_2(m)$, $t_1=T_1(m)$. This reads
\begin{equation}
s_1+s_2-t_1-t_2= i_m-i_{m-1}+1,
\end{equation}
which proves equality \eqref{eq: deltac_Cn4}.
\end{itemize}

\item Case $i_j=\ell-1$. In this case we have to compute
\begin{equation}
\tau_{\ell-1}=-\delta^c_{\ell-2}+2\delta^c_{\ell-1}-2\delta^c_{\ell},
\end{equation}
and the explicit computations, together with equalities \eqref{eq:gammaC} and \eqref{eq: formulacncoefficients}, prove \eqref{eq: xi_Cn5} and \eqref{eq: xi_Cn6}.

We have to consider two cases.
\begin{itemize}

\item Case $\ell\in S$. We will prove that
\begin{equation}\label{eq: deltac_Cn5}
\tau_{\ell-1}= 3-\ell+i_{j-1}+ 2(S_2(j)-T_1(j)).
\end{equation}
The roots in $\Delta^c_+\setminus\spann\lrbr{\gamma_i \vert i\in S^c}$ such that $\alpha_{\ell-1}\neq 0$ are
\begin{align}
& \sum_{k=n_1}^{\ell-1}\gamma_k,\quad \lvert S\cap\lrbr{n_1,\ldots,\ell-1}\rvert \in 2\ZZ \label{eq: n-1An}\\
& \sum_{k=n_1}^{\ell-1}\gamma_k+\gamma_{\ell},\quad \lvert S\cap\lrbr{n_1,\ldots,\ell}\rvert  \in 2\ZZ \label{eq: n-1An2} \\
& \sum_{k=n_1}^{n_2}\gamma_k+\sum_{k=n_2+1}^{\ell-1}2\gamma_k+\gamma_{\ell},\quad \lvert S\cap\lrbr{n_1,\ldots,n_2}\rvert  \in 1+2\ZZ, \label{eq: n-1Bn}
\end{align}
with coefficient $\alpha_{\ell-1}=1$, $\alpha_{\ell-1}=1$ and $\alpha_{\ell-1}=2$ respectively. Call $L=\sum_{i=1}^{\ell} \lambda_i\gamma_i$ the sum of all roots of type \eqref{eq: n-1An}, $M=\sum_{i=1}^{\ell} \mu_i\gamma_i$ the sum of roots of type \eqref{eq: n-1An2} and $N=\sum_{i=1}^{\ell} \nu_i\gamma_i$ the sum of the roots of type \eqref{eq: n-1Bn}. Note that, $\delta^c_{\ell-1}=\lambda_{\ell-1}+\mu_{\ell-1}+\nu_{\ell-1}$. On the other hand, the roots in $\Delta^c_+\setminus\spann\lrbr{\gamma_i \vert i\in S^c}$ such that $\alpha_{\ell}\neq 0$ are \eqref{eq: n-1An2} and \eqref{eq: n-1Bn}, both with coefficient $\alpha_{\ell}=1$. In particular, $\delta^c_{\ell}=\mu_{\ell-1}+\frac{1}{2}\nu_{\ell-1}$. So
\begin{equation}
\begin{split}
2\delta^c_{\ell-1}-2\delta^c_{\ell}=& 2\lambda_{\ell-1}+2\mu_{\ell-1}+2\nu_{\ell-1}-2\mu_{\ell-1}-\nu_{\ell-1}\\
=& \lambda_{\ell-1}+\mu_{\ell-1}+\nu_{\ell-1}+\lambda_{\ell-1}-\mu_{\ell-1}\\
=& \delta^c_{\ell-1}+\lambda_{\ell-1}-\mu_{\ell-1}, 
\end{split}
\end{equation}
which reads
\begin{equation}
\tau_{\ell-1}= s_1-t_1+\lambda_{\ell-1}-\mu_{\ell-1}.
\end{equation}
So it suffices to compute the coefficients $s_1, t_1, \lambda_{\ell-1},\mu_{\ell-1}$. We start with $\lambda_{\ell-1}$. This coefficient is given by the $(\ell-1)$th coefficient of the sum of the roots of type \eqref{eq: n-1An}, that corresponds to
\begin{equation}
\lambda_{\ell-1}=\sum_{k=1}^{\lfloor\frac{j}{2} \rfloor} i_{j-2k+1}-i_{j-2k}=S_2(j).
\end{equation}
Similarly, the coefficient $\mu_{\ell-1}$ corresponds to 
\begin{equation}
\mu_{\ell-1}=\sum_{k=1}^{\lfloor\frac{j+1}{2} \rfloor-1} i_{j-2k}-i_{j-2k-1}+(i_j-i_{j-1}-1)=i_j-i_{j-1}-1+T_1(j).
\end{equation}
Then, similar computations as in the other sections reads $s_1=S_2(j)+1$, since we need to take into account also the root $\gamma_{\ell-1}+\gamma_{\ell}$, and and $t_1=T_1(j)$. Finally
\begin{equation}
s_1-t_1+\lambda_{\ell-1}-\mu_{\ell-1}= 3-\ell+i_{j-1}+ 2(S_2(j)-T_1(j)).
\end{equation}
This proves equality \eqref{eq: deltac_Cn5}.

\item Case $\ell\notin S$. We will prove that 
\begin{equation}\label{eq: deltac_Cn6}
\tau_{\ell-1}= \ell-i_{j-1}.
\end{equation}
The roots in $\Delta^c_+\setminus\spann\lrbr{\gamma_i \vert i\in S^c}$ such that $\alpha_{\ell-1}\neq 0$ are
\begin{align}
& \sum_{k=n_1}^{\ell-1}\gamma_k,\quad \lvert S\cap\lrbr{n_1,\ldots,\ell-1}\rvert \in 2\ZZ \label{eq: n-1Anf}\\
& \sum_{k=n_1}^{\ell-1}\gamma_k+\gamma_{\ell},\quad \lvert S\cap\lrbr{n_1,\ldots,\ell-1}\rvert \in 2\ZZ \label{eq: n-1An2f} \\
& \sum_{k=n_1}^{n_2}\gamma_k+\sum_{k=n_2+1}^{\ell-1}2\gamma_k+\gamma_{\ell},\quad \lvert S\cap\lrbr{n_1,\ldots,n_2}\rvert \in 2\ZZ, \label{eq: n-1Bnf}
\end{align}
with coefficients $\alpha_{\ell-1}=1$, $\alpha_{\ell-1}=1$ and $\alpha_{\ell-1}=2$ respectively. Similarly as above, it turns out that $2\delta^c_{\ell}=\delta^c_{\ell-1}$, hence $2\delta^c_{\ell-1}-2\delta^c_{\ell}=\delta^c_{\ell-1}$. This, in turns, gives
\begin{equation}
\tau_{\ell-1}= -\delta^c_{\ell-1}+s_1-t_1+\delta^c_{\ell-1}=s_1-t_1.
\end{equation}
So it is enough to compute the coefficients $s_1, t_1$, and this can be done with the techniques used for the previous points and proofs. Hence
\begin{equation}
\begin{split}
s_1=&\sum_{k=1}^{\lfloor\frac{j+1}{2} \rfloor-1} i_{j-2k}-i_{j-2k-1}+(i_j-i_{j-1}-1)+2=T_1(j)+i_j-i_{j-1}+1\\
t_1=&\sum_{k=1}^{\lfloor\frac{j+1}{2} \rfloor-1} i_{j-2k}-i_{j-2k-1}=T_1(j),
\end{split}
\end{equation}
and this implies
\begin{equation}
s_1-t_1 = \ell-i_{j-1}.
\end{equation}
This proves equality \eqref{eq: deltac_Cn6} and concludes the proof of the theorem.
\end{itemize}
\end{itemize}
\end{proof}

The next section is the last one and concludes the computation of the paper.


\subsection{$D_{\ell}$ family}\label{sec:dnCoefficients}

In this final section we are going to study the coefficients $\xi_i$ for Vogan diagrams of type $D_{\ell}$. Consider a Vogan diagram of type $D_{\ell}$ with $S=\lrbr{i_1,\ldots,i_m}$ the set of indices of simple non-compact roots and denote by $\{\gamma_1,\ldots,\gamma_{\ell}\}$ the set of simple roots. As in the previous cases, with the definition of the Cartan matrix $A$ for Lie algebras of type $D_{\ell}$ \cite[Sec.~11.4,Tab.~1]{Humphreys1972}, we get
\begin{equation}\label{eq: gammadn}
\begin{split}
\Gamma_{i_k}=&\left( A\sum_{\alpha\in\spann\lrbr{\gamma_i \vert i\in S^c}}\alpha\right)_{i_k}\\
=&\left\{ \begin{array}{lcl}
-\dime(\hh_{i_{k}})-\dime(\hh_{i_{k+1}})-\delta_{\nu\in S}\delta_{\nu'\notin S}\delta_{k=m-1} & \text{if} &k\neq m,\ i_k\neq \ell-2\\
-\dime(\hh_{i_{m}})-2\dime(\hh_{i_{m+1}})+2& \text{if} &k= m, i_k\neq \ell-2 \\
-\dime(\hh_{i_{k}})-(\delta_{\ell-1\notin S}+\delta_{\ell\notin S}) & \text{if} &i_k= \ell-2\\
-2\dime(\hh_{i_{k}}) & \text{if} & i_k=\ell-1,\ell
 \end{array}\right.\\
=&\left\{ \begin{array}{lcl}
i_{k-1}-i_{k+1}+2-\delta_{\nu\in S}\delta_{\nu'\notin S}\delta_{k=m-1} & \text{if} & k\neq m,\ i_k\neq \ell-2\\
-2\ell+i_m+ i_{m-1}+3 & \text{if} & k= m, i_k\neq \ell-2 \\
-\ell+i_{k-1}+3-\delta_{\ell-1\notin S}-\delta_{\ell\notin S} & \text{if} & i_k= \ell-2\\
-2\ell+2i_{k-1}+4 & \text{if} & i_k=\ell-1,\ell
\end{array}\right.,
\end{split}
\end{equation}
with $\nu$ and $\nu'$ equal to $\ell-1$, $\ell$ or $\ell$, $\ell-1$. The $\delta$ factor in the first case comes from the fact that if $k=m-1$ and $i_m=\nu$, then there is another root that contributes to the coefficient, that is $\sum_{i=i_k+1}^{\ell-1}\gamma_i+\gamma_{\nu'}$. The same holds for $\nu'$. For the third case the second summand comes from the fact that the roots of the Lie subalgebra with Dynkin diagram having vertices with indices greater than $\ell-2$ are only: $\gamma_{\ell-1}$ if $\gamma_{\ell}\in S$, $\gamma_{\ell}$ if $\gamma_{\ell-1}\in S$, $\gamma_{\ell-1}+\gamma_{\ell}$ if $\gamma_{\ell-1},\gamma_{\ell}\notin S$, none if $\gamma_{\ell-1},\gamma_{\ell}\in S$. The other terms can be computed as in previous families. Moreover
\begin{equation}\label{eq: xiCoefficientsDnPre}
\xi_{i_k}=-4-2\Gamma_{i_k}+4\left( A\sum_{\alpha\in\Delta_+^c}\alpha\right)_{i_k}=2\left(-2+\Gamma_{i_k}+2\tau_{i_k}\right),
\end{equation}
where $\delta^c=\sum_{\alpha\in\Delta_+^c\setminus\spann\lrbr{\gamma_i\vert i\in S^c}}\alpha$ and $\tau_{i_k}:=\left(A\delta^c\right)_{i_k}$. As a consequence of the structure of the Cartan matrix $A$, the coefficients $\tau_{i_k}$ can be written explicitly as
\begin{equation}
\tau_{i_k}=\left\{\begin{array}{lcl}
-\delta^c_{i_k-1}+2\delta^c_{i_k}-\delta^c_{i_k+1} & \text{if} & i_k\neq \ell-2,\ell-1,\ell \\
-\delta^c_{\ell-3}+2\delta^c_{\ell-2}-\delta^c_{\ell-1}-\delta^c_{\ell} & \text{if} & i_k=\ell-2\\
-\delta^c_{\ell-2}+2\delta^c_{\ell-1} & \text{if} & i_j=\ell-1,\ell\\
\end{array}\right..
\end{equation}
Note that the coefficients $\tau_{\ell-1}$ and $\tau_{\ell}$ are the same, by symmetry of Dynkin diagrams of type $D_{\ell}$.

Again, the goal is to compute explicitly the coefficients $\tau_{i_k}$ in terms of the indices of the simple non-compact roots of the Vogan diagram. Since the computations depend on the structure of the roots of type $D_{\ell}$, we recall them here below. The roots for Lie algebra of type $D_{\ell}$ are
\begin{equation}
\begin{split}
\Delta_+=&\left\{ \sum_{i=n_1}^{n_2}\gamma_i\ \vert\ 1\leq n_1\leq n_2\leq \ell-2\right\}\cup \left\{ \sum_{i=n_1}^{\ell-2}\gamma_i+\gamma_{j}\ \vert\ 1\leq n_1\leq \ell-2,\ j=\ell-1,\ell\right\} \cup\\ 
& \cup \left\{ \sum_{i=n_1}^{\ell-2}\gamma_i+\gamma_{\ell-1}+\gamma_{\ell}\ \vert\ 1\leq n_1\leq \ell-2\right\}\cup \\ & \cup \left\{ \sum_{i=n_1}^{n_2}\gamma_i+\sum_{i=n_2+1}^{\ell-2}2 \gamma_i+\gamma_{\ell-1}+\gamma_{\ell}\ \vert\ 1\leq n_1\leq n_2\leq \ell-3\right\} 
\end{split}
\end{equation}
and the positive compact roots are
\begin{equation}\label{eq: compactRootsDn}
\begin{split}
\Delta^c_+=&\left\{ \sum_{i=n_1}^{n_2}\gamma_i\ \vert\ 1\leq n_1\leq n_2\leq \ell-2,\ \lvert S\cap\{n_1,\ldots,n_2\}\rvert\in 2\ZZ\right\}\cup \\ 
& \cup \left\{ \sum_{i=n_1}^{\ell-2}\gamma_i+\gamma_{j}\ \vert\ 1\leq n_1\leq \ell-1,\ j=\ell-1,\ell,\ \lvert S\cap\{n_1,\ldots,\ell-2,j\}\rvert\in 2\ZZ\right\}\cup \\
& \cup \left\{ \sum_{i=n_1}^{\ell-2}\gamma_i+\gamma_{\ell-1}+\gamma_{\ell}\ \vert\ 1\leq n_1\leq \ell-2,\ \lvert S\cap\{n_1,\ldots,\ell-2,\ell-1,\ell\}\rvert\in 2\ZZ\right\}\cup\\
& \cup\left\{ \sum_{i=n_1}^{n_2}\gamma_i+\sum_{i=n_2+1}^{\ell-2}2 \gamma_i+\gamma_{\ell-1}+\gamma_{\ell}\ \vert\ 1\leq  n_1\leq n_2\leq \ell-3,\ \lvert \right.\\ &\phantom{\cup}\lvert S\cap\{n_1,\ldots,n_2,\ell-1,\ell\}\rvert\in 2\ZZ\Bigg{\}}\\
\end{split}
\end{equation}

At this point we can write explicitly the coefficients $\xi_i$ in term of the indices in $S$.

\begin{thm}
Given a Vogan diagram of type $D_{\ell}$ with $S=\lrbr{i_1,\ldots,i_m}$ the set of indices of simple non-compact roots, for $i_j\in S$ one has

\begin{itemize}

\item If $i_j\notin\lrbr{\ell-2,\ell-1,\ell}$
\begin{itemize}

\item If $j\neq m$ and $\lvert S\cap \lrbr{\ell-1,\ell}\rvert=0,2$
\begin{equation}\label{eq: xi_Dn1}
\xi_{i_j}= 2(-i_{j-1}+4 i_j-3i_{j+1}+4(S_+(j)+(-1)^{m+j})
\end{equation}

\item If $j\neq m$ and $\lvert S\cap \lrbr{\ell-1,\ell}\rvert=1$
\begin{equation}\label{eq: xi_Dn2}
\xi_{i_j}= 2(3i_{j-1}-4 i_j+i_{j+1}+4S_-(j) +\delta_{j,m-1})
\end{equation}

\item If $j= m$
\begin{equation}\label{eq: xi_Dn3}
\xi_{i_m}= 2( -2\ell+3i_m-i_{m-1}-1)
\end{equation}
\end{itemize}

\item If $i_j=\ell-2$
\begin{itemize}
\item If $\lvert S\cap \lrbr{\ell-1,\ell}\rvert=0,2$
\begin{equation}\label{eq: xi_Dn4}
\xi_{\ell-2}=2(\ell-i_{j-1}-5+2(2\delta_{i_{m-1},\ell-1}\delta_{i_m,\ell}-\delta_{\ell-1\notin S}\delta_{\ell\notin S}))
\end{equation}

\item If $\lvert S\cap \lrbr{\ell-1,\ell}\rvert=1$
\begin{equation}\label{eq: xi_Dn5}
\xi_{\ell-2}= 2(-3\ell+3i_{j-1}+8+4S_-(j))
\end{equation}
\end{itemize}

\item If $i_j=\ell-1,\ell$
\begin{itemize}

\item If $\lvert S\cap \lrbr{\ell-1,\ell}\rvert=1$
\begin{equation}\label{eq: xi_Dn6}
\xi_{i_j}=2(-2\ell+2i_{j-1}+2+4S_-(j))
\end{equation}

\item If $\lvert S\cap \lrbr{\ell-1,\ell}\rvert=2$
\begin{equation}\label{eq: xi_Dn7}
\xi_{i_j}=2(\ell-4-i_{m-2})
\end{equation}
\end{itemize}
\end{itemize}
\end{thm}

\begin{proof}
As for the other proofs assume that $i_0=0$ and $i_{m+1}=\ell+1$. 

\begin{itemize}
\item Case $i_j\notin\lrbr{\ell-2,\ell-1,\ell}$.
 \begin{itemize}
 \item Case $j\neq m$ and $\lvert S\cap \lrbr{\ell-1,\ell}\rvert=0,2$. We will prove that
\begin{equation}\label{eq: dnformula1}
\tau_{i_j}= -i_{j-1}+2 i_j-i_{j+1}+2(S_1(j)-T_2(j))+2(-1)^{m+j},
\end{equation}
that, together with the coefficients \eqref{eq: gammadn} and \eqref{eq: xiCoefficientsDnPre} will prove identity \eqref{eq: xi_Dn1}. Recall that 
\begin{equation}
\tau_{i_j}=-\delta_{i_j-1}^c+2\delta^c_{i_j}-\delta_{i_j+1}^c=s_1-t_1+s_2-t_2.
\end{equation}
Similarly as before, by looking at the compact roots listed in \eqref{eq: compactRootsDn} one has
\begin{equation}
\begin{split}
s_1=& \sum_{k=1}^{\lceil\frac{m -j}{2} \rceil}i_{j+2k}-i_{j+2k-1}+\delta_{\ell+1,i_{j+2\lceil\frac{m -j}{2} \rceil}}+(i_j-i_{j-1}-1)+ \\ &+\sum_{k=1}^{\lfloor\frac{j+1}{2} \rfloor-1} i_{j-2k}-i_{j-2k-1} + \sum_{k=1}^{\lceil\frac{m -j}{2} \rceil}i_{j+2k}-i_{j+2k-1}-3\delta_{\ell+1,i_{j+2\lceil\frac{m -j}{2} \rceil}}\\
=& i_j-i_{j-1}-1+2S_1(j)+T_1(j)+\delta_{\ell+1,i_{j+2\lceil\frac{m -j}{2} \rceil}}-2\delta_{\ell+1,i_{j+2\lceil\frac{m -j}{2} \rceil}}.
\end{split}
\end{equation} 
The coefficient $\delta_{\ell+1,i_{j+2\lceil\frac{m -j}{2} \rceil}}$ appears since the root $\sum_{l=i_j}^{\ell-2}\gamma_l+\gamma_{\ell-1}+\gamma_{\ell}$ gives contribution only if $|\lrbr{i_{j+1},\ldots,i_m}|$ is even. Similarly, $3\delta_{\ell+1,i_{j+2\lceil\frac{m -j-1}{2} \rceil+1}}$ appears to avoid overcounting the roots of type $\sum_{l=i_j}^{n_2}\gamma_l+\sum_{l=n_2+1}^{\ell-2}2\gamma_l+\gamma_{\ell-1}+\gamma_{\ell}$ with $|S\cap\lrbr{i_j,\ldots,\ell-2}|\in 2\ZZ$.

For $s_2$ and $t_1$ we have $s_2=S_2(j)$, $t_2=T_2(j)$. Finally, for $t_2$, similarly as for $s_1$ we have
\begin{equation}
\begin{split}
t_2= & \sum_{k=1}^{\lceil\frac{m -j-1}{2} \rceil}i_{j+2k+1}-i_{j+2k}+\delta_{\ell+1,i_{j+2\lceil\frac{m -j-1}{2} \rceil+1}}+(i_{j+1}-i_j-1)+ \\ & +\sum_{k=1}^{\lceil\frac{m -j-1}{2} \rceil}i_{j+2k+1}-i_{j+2k} -3\delta_{\ell+1,i_{j+2\lceil\frac{m -j-1}{2} \rceil+1}} + \sum_{k=1}^{\lfloor\frac{j}{2} \rfloor} i_{j-2k+1}-i_{j-2k}\\
=&  i_{j+1}-i_j-1 +2T_2(j)+S_2(j)-2\delta_{\ell+1,i_{j+2\lceil\frac{m -j-1}{2} \rceil+1}}.
\end{split}
\end{equation}
Hence
\begin{equation}
\begin{split}
s_1-t_1+s_2-t_2=& -i_{j-1}+2i_j-i_{j-1} +2(S_1(j)-T_2(j))+2(-1)^{m+j},
\end{split}
\end{equation}
This proves \eqref{eq: dnformula1}.

\item Case $j\neq m$ and $\lvert S\cap \lrbr{\ell-1,\ell}\rvert=1$. We have to prove that
\begin{equation}\label{eq: dnformula2}
\tau_{i_j}= i_{j-1}-2 i_j+i_{j+1}+2(S_2(j)-T_1(j)) +\delta_{j,m-1}.
\end{equation}
This, together with the coefficients \eqref{eq: gammadn} and \eqref{eq: xiCoefficientsDnPre} will prove identity \eqref{eq: xi_Dn2}. Without loss of generality we can assume $\ell-1\in S$, since the same result holds for $\ell\in S$, by symmetry. We need to compute
 \begin{equation}
\tau_{i_j}=-\delta_{i_j-1}^c+2\delta^c_{i_j}-\delta_{i_j+1}^c=s_1-t_1+s_2-t_2. 
\end{equation}
The computations for $s_1$, $s_2$ and $t_1$ are similar as above, hence
\begin{equation}
\begin{split}
s_1=& \sum_{k=1}^{\lceil\frac{m -j}{2} \rceil}i_{j+2k}-i_{j+2k-1}+ \delta_{\ell+1,i_{j+2\lceil \frac{m-j-1}{2}\rceil+1}}  +\sum_{k=1}^{\lceil\frac{m -j-1}{2} \rceil}i_{j+2k+1}-i_{j+2k}+  \\
& -\delta_{\ell+1,i_{j+2\lceil \frac{m-j}{2}\rceil }} -2\delta_{\ell+1,i_{j+2\lceil \frac{m-j-1}{2}\rceil +1}} + (i_{j+1}-i_j)+ \sum_{k=1}^{\lfloor\frac{j}{2} \rfloor} i_{j-2k+1}-i_{j-2k} \\
=& i_{j+1}-i_j+S_1(j)+T_2(j)+S_2(j)-\delta_{\ell+1,i_{j+2\lceil \frac{m-j}{2}\rceil} } -\delta_{\ell+1,i_{j+2\lceil \frac{m-j-1}{2}\rceil+1}},
\end{split}
\end{equation}
and $s_2=S_2(j)$, $t_1=T_1(j)$. For $t_2$ we need to count the compact roots in $\Delta^c_+\setminus\spann\lrbr{\gamma_i \vert i\in S^c}$ such that $\alpha_{i_j+1}-\alpha_{i_j}=1$. In order to do that we can proceed as for the other cases, obtaining
\begin{equation}
\begin{split}
t_2=& \sum_{k=1}^{\lceil\frac{m -j-1}{2} \rceil}i_{j+2k+1}-i_{j+2k}+\delta_{\ell+1,i_{j+2\lceil \frac{m-j}{2}\rceil}}+\sum_{k=1}^{\lceil\frac{m -j}{2} \rceil}i_{j+2k}-i_{j+2k-1}+ \\ &-\delta_{j,m-1} -2\delta_{\ell+1,i_{j+2\lceil \frac{m-j}{2}\rceil}}-\delta_{\ell+1,i_{j+2\lceil \frac{m-j-1}{2}\rceil+1}}+(i_j-i_{j-1})+ \\ &+\sum_{k=1}^{\lfloor\frac{j+1}{2} \rfloor-1} i_{j-2k}-i_{j-2k-1}\\
=& i_j-i_{j-1}+T_2(j)+S_1(j)+T_1(j) -\delta_{\ell+1,i_{j+2\lceil \frac{m-j}{2}\rceil}} +\\ &-\delta_{\ell+1,i_{j+2\lceil \frac{m-j-1}{2}\rceil+1}}-\delta_{j,m-1}.
\end{split}
\end{equation}
In this case we see a $\delta_{j,m-1}$ as last summand. It comes from the fact that, when $j=m-2$, the contribute from $T_2(j)+\delta_{\ell+1,i_{j+2\lceil \frac{m-j}{2}\rceil}}$ should be $0$, since $\gamma_{m-2}+\ldots+\gamma_{\ell-2}+\gamma_{\ell}\notin \Delta^c_+\setminus\spann\lrbr{\gamma_i \vert i\in S^c}$. However, this contribute is $1$, since $\delta_{\ell+1,i_{j+2\lceil \frac{m-j}{2}\rceil}}=1$, so we need to subtract $1$ in order to get the correct number of roots. Summing all the contributes we get
\begin{equation}
s_1-t_1+s_2-t_2=i_{j-1}-2i_j+i_{j+1}+2(S_2(j)-T_1(j))+\delta_{j,m-1}.
\end{equation}
This proves formula \eqref{eq: dnformula2}.

\item Case $j=m$. We prove
\begin{equation}\label{eq: dnformula3}
\tau_{i_m}= i_m-i_{m-1}-1,
\end{equation}
that implies identity \eqref{eq: xi_Dn2}, together with the coefficients \eqref{eq: gammadn} and \eqref{eq: xiCoefficientsDnPre}. In this case we have 
 \begin{equation}
\tau_{i_j}=-\delta_{i_j-1}^c+2\delta^c_{i_j}-\delta_{i_j+1}^c=s_1-t_1+s_2-t_2,
\end{equation}
and the computations are as before. Hence
\begin{equation}
s_1=(i_j-i_{j-1}-1)+\sum_{k=1}^{\lfloor\frac{j+1}{2} \rfloor-1} i_{j-2k}-i_{j-2k-1}=i_j-i_{j-1}-1+T_1(j),
\end{equation}
$s_2=S_2(j)$, $t_1=T_1(j)$ and 
\begin{equation}
t_2=\sum_{k=1}^{\lfloor\frac{j}{2} \rfloor} i_{j-2k+1}-i_{j-2k}=S_2(j).
\end{equation}
Thus
\begin{equation}
s_1-t_1+s_2-t_2=i_m-i_{m-1}-1,
\end{equation}
and this proves \eqref{eq: dnformula3}.
\end{itemize}

\item Case $i_j=\ell-2$. In this case we need to compute 
\begin{equation}
\tau_{i_j}=-\delta^c_{\ell-3}+2\delta^c_{\ell-2}-\delta^c_{\ell-1}-\delta^c_{\ell}.
\end{equation}
There are two cases to consider.

\begin{itemize}
\item Case $\lvert S\cap \lrbr{\ell-1,\ell}\rvert=0,2$. We will prove that
\begin{equation}\label{eq: dnformula4}
\tau_{\ell-2}=\ell-3-i_{j-1}+2\delta_{i_{m-1},\ell-1}\delta_{i_m,\ell},
\end{equation}
that implies identity \eqref{eq: xi_Dn4}, together with the coefficients \eqref{eq: gammadn} and \eqref{eq: xiCoefficientsDnPre}. This point is a little technical, so we will illustrate most of the computations, which will also be used in the subsequent points. The compact roots in $\Delta^c_+\setminus\spann\lrbr{\gamma_i \vert i\in S^c}$ such that $\alpha_{\ell-2}\neq 0$ are
\begin{align}
&\sum_{i=n_1}^{\ell-2}\gamma_i,\quad |S\cap\lrbr{n_1,\ldots,\ell-2}|\in 2\ZZ \\
&\sum_{i=n_1}^{\ell-2}\gamma_i+\nu,\ \nu=\gamma_{\ell-1},\gamma_{\ell},\ |S\cap\lrbr{n_1,\ldots,\ell-2}|\in 1+2\ZZ, \text{if}\ \gamma_{\ell-1},\gamma_{\ell}\in S \\
&\sum_{i=n_1}^{\ell-2}\gamma_i+\nu,\ \nu=\gamma_{\ell-1},\gamma_{\ell},\quad |S\cap\lrbr{n_1,\ldots,\ell-2}|\in 2\ZZ, \text{if}\ \gamma_{\ell-1},\gamma_{\ell}\notin S \\
&\sum_{i=n_1}^{\ell-2}\gamma_i+\gamma_{\ell-1}+\gamma_{\ell},\quad |S\cap\lrbr{n_1,\ldots,\ell-2}|\in 2\ZZ \\
&\sum_{i=n_1}^{n_2}\gamma_i+\sum_{i=n_2+1}^{\ell-2}2\gamma_i+\gamma_{\ell-1}+\gamma_{\ell},\quad |S\cap\lrbr{n_1,\ldots,n_2}|\in 2\ZZ, 
\end{align}
with coefficients $\alpha_{\ell-2}=1,1,1,1,2$ respectively. On the other hand, the compact roots in $\Delta^c_+\setminus\spann\lrbr{\gamma_i \vert i\in S^c}$ such that $\alpha_{\ell-1}\neq 0$ or $\alpha_{\ell}\neq 0$ are
\begin{align}
&\sum_{i=n_1}^{\ell-2}\gamma_i+\nu,\ \nu=\gamma_{\ell-1},\gamma_{\ell},\ |S\cap\lrbr{n_1,\ldots,\ell-2}|\in 1+2\ZZ, \text{if}\ \gamma_{\ell-1},\gamma_{\ell}\in S \\
&\sum_{i=n_1}^{\ell-2}\gamma_i+\nu,\ \nu=\gamma_{\ell-1},\gamma_{\ell},\quad |S\cap\lrbr{n_1,\ldots,\ell-2}|\in 2\ZZ, \text{if}\ \gamma_{\ell-1},\gamma_{\ell}\notin S \\
&\sum_{i=n_1}^{\ell-2}\gamma_i+\gamma_{\ell-1}+\gamma_{\ell},\quad |S\cap\lrbr{n_1,\ldots,\ell-2}|\in 2\ZZ \\
&\sum_{i=n_1}^{n_2}\gamma_i+\sum_{i=n_2+1}^{\ell-2}2\gamma_i+\gamma_{\ell-1}+\gamma_{\ell},\quad |S\cap\lrbr{n_1,\ldots,n_2}|\in 2\ZZ, 
\end{align}
with coefficients $0$ or $1$. Moreover, the compact roots in $\Delta^c_+\setminus\spann\lrbr{\gamma_i \vert i\in S^c}$ such that $\alpha_{\ell-3}\neq 0$ are 
\begin{align}
&\sum_{i=n_1}^{\ell-3}\gamma_i,\quad |S\cap\lrbr{n_1,\ldots,\ell-3}|\in 2\ZZ,\ n_1\leq \ell-3 \\
&\sum_{i=n_1}^{\ell-2}\gamma_i,\quad |S\cap\lrbr{n_1,\ldots,\ell-2}|\in 2\ZZ,\ n_1\leq \ell-3 \\
&\sum_{i=n_1}^{\ell-2}\gamma_i+\nu,\ \nu=\gamma_{\ell-1},\gamma_{\ell},\quad |S\cap\lrbr{n_1,\ldots,\ell-2}|\in 1+2\ZZ, \\ & \phantom{spazio}\text{if}\ \gamma_{\ell-1},\gamma_{\ell}\in S,\ n_1\leq \ell-3 \\
&\sum_{i=n_1}^{\ell-2}\gamma_i+\nu,\ \nu=\gamma_{\ell-1},\gamma_{\ell},\quad |S\cap\lrbr{n_1,\ldots,\ell-2}|\in 2\ZZ,\\ & \phantom{spazio}\text{if}\ \gamma_{\ell-1},\gamma_{\ell}\notin S,\ n_1\leq \ell-3 \\
&\sum_{i=n_1}^{\ell-2}\gamma_i+\gamma_{\ell-1}+\gamma_{\ell},\quad |S\cap\lrbr{n_1,\ldots,\ell-2}|\in 2\ZZ,\ n_1\leq \ell-3 \\
&\sum_{i=n_1}^{n_2}\gamma_i+\sum_{i=n_2+1}^{\ell-2}2\gamma_i+\gamma_{\ell-1}+\gamma_{\ell},\quad |S\cap\lrbr{n_1,\ldots,n_2}|\in 2\ZZ,\ n_2\leq \ell-4 \\
&\sum_{i=n_1}^{\ell-3}\gamma_i+2\gamma_{\ell-2}+\gamma_{\ell-1}+\gamma_{\ell},\quad |S\cap\lrbr{n_1,\ldots,\ell-3}|\in 2\ZZ,\ n_1\leq \ell-3.
\end{align}
Not all the roots contribute to the summation. In particular, the ones that do are
\begin{itemize}
\item $\sum_{i=n_1}^{\ell-3}\gamma_i,\ |S\cap\lrbr{n_1,\ldots,\ell-3}|\in 2\ZZ,\ n_1\leq \ell-3$. These contribute to $-\delta^c_{\ell-3}+2\delta^c_{\ell-2}-\delta^c_{\ell-1}-\delta^c_{\ell}$ only for the $-\delta^c_{\ell-3}$ part, since the other $\delta^c_{\ell-2},\delta^c_{\ell-1},\delta^c_{\ell}$ give contribution $0$. So the given contribute is $-\sum_{k=1}^{\lfloor\frac{j+1}{2} \rfloor-1} i_{j-2k}-i_{j-2k-1}$.
\item $\sum_{i=n_1}^{\ell-2}\gamma_i+\gamma_{\ell-1}+\gamma_{\ell},\ |S\cap\lrbr{n_1,\ldots,\ell-2}|\in 2\ZZ,\ n_1\leq \ell-3$. These contribute to $-\delta^c_{\ell-3}+2\delta^c_{\ell-2}-\delta^c_{\ell-1}-\delta^c_{\ell}$ only for the $-\delta^c_{\ell-3}$ part, since $\delta^c_{\ell-2},\delta^c_{\ell-1},\delta^c_{\ell}$ factors cancel out. So the given contribute is $-\sum_{k=1}^{\lfloor\frac{j}{2} \rfloor} i_{j-2k+1}-i_{j-2k}$.
\item $\sum_{i=n_1}^{\ell-2}\gamma_i,\ |S\cap\lrbr{n_1,\ldots,\ell-2}|\in 2\ZZ,\ n_1\leq \ell-3$. These contribute to $-\delta^c_{\ell-3}+2\delta^c_{\ell-2}-\delta^c_{\ell-1}-\delta^c_{\ell}$ only for the $-\delta^c_{\ell-3}+2\delta^c_{\ell-2}$ part, since $\delta^c_{\ell-1},\delta^c_{\ell}$ give contribution $0$. So the given contribute is $\sum_{k=1}^{\lfloor\frac{j}{2} \rfloor} i_{j-2k+1}-i_{j-2k}$.
\item $\sum_{i=n_1}^{n_2}\gamma_i+\sum_{i=n_2+1}^{\ell-2}2\gamma_i+\gamma_{\ell-1}+\gamma_{\ell},\ |S\cap\lrbr{n_1,\ldots,n_2}|\in 2\ZZ,\ n_2\leq \ell-3$. These contribute to the whole coefficient $-\delta^c_{\ell-3}+2\delta^c_{\ell-2}-\delta^c_{\ell-1}-\delta^c_{\ell}$, by $+(\ell-2-i_{j-1}-1)+\sum_{k=1}^{\lfloor\frac{j+1}{2} \rfloor-1} i_{j-2k}-i_{j-2k-1}$.
\item The roots $\gamma_{\ell-2}+\gamma_{\ell-1}$, $\gamma_{\ell-2}+\gamma_{\ell}$, if $\gamma_{\ell-1},\gamma_{\ell}\in S$. These contribute to $-\delta^c_{\ell-3}+2\delta^c_{\ell-2}-\delta^c_{\ell-1}-\delta^c_{\ell}$ only for the $2\delta^c_{\ell-2}-\delta^c_{\ell-1}-\delta^c_{\ell}$ part, with a total contribute of $2$.
\end{itemize}
So 
\begin{equation}
\tau_{\ell-2}= \ell-3-i_{j-1}+2\delta_{i_{m-1},\ell-1}\delta_{i_m,\ell},
\end{equation}
amd this proves equality \eqref{eq: dnformula4}.

\item Case $\lvert S\cap \lrbr{\ell-1,\ell}\rvert=1$. We prove that
\begin{equation}\label{eq: dnformula5}
\tau_{\ell-2}= -i_j+i_{j-1}+2+2(S_2(j)-T_1(j))
\end{equation}
that implies identity \eqref{eq: xi_Dn5}, together with the coefficients \eqref{eq: gammadn} and \eqref{eq: xiCoefficientsDnPre}. Without loss of generality we can assume that $\gamma_m=\gamma_{\ell-1}$, and the proof in this case is similar to the one of the previous point. In particular, the roots that contribute in the summation are
\begin{itemize}
\item $\sum_{i=n_1}^{\ell-3}\gamma_i,\ |S\cap\lrbr{n_1,\ldots,\ell-3}|\in 2\ZZ,\ n_1\leq \ell-3$, which contribute by $-\sum_{k=1}^{\lfloor\frac{j+1}{2} \rfloor-1} i_{j-2k}-i_{j-2k-1}$.
\item $\sum_{i=n_1}^{\ell-2}\gamma_i,\ |S\cap\lrbr{n_1,\ldots,\ell-2}|\in 2\ZZ,\ n_1\leq \ell-3$, which contribute by $\sum_{k=1}^{\lfloor\frac{j}{2} \rfloor} i_{j-2k+1}-i_{j-2k}$.
\item $\gamma_{\ell-2}+\gamma_{\ell-1}$, which contributes by a factor $1$.
\item $\sum_{i=n_1}^{\ell-2}\gamma_i+\gamma_{\ell-1}+\gamma_{\ell},\ |S\cap\lrbr{n_1,\ldots,\ell-2}|\in 1+2\ZZ,\ n_1\leq \ell-3$, which contribute by $-(\ell-1-i_{j-1}-1)-\sum_{k=1}^{\lfloor\frac{j+1}{2} \rfloor-1} i_{j-2k}-i_{j-2k-1}$.
\item $\sum_{i=n_1}^{\ell-3}\gamma_i+2\gamma_{\ell-2}+\gamma_{\ell-1}+\gamma_{\ell},\ |S\cap\lrbr{n_1,\ldots,\ell-3}|\in 1+2\ZZ,\ n_1\leq \ell-3$, which contribute by $\sum_{k=1}^{\lfloor\frac{j}{2} \rfloor} i_{j-2k+1}-i_{j-2k}$.
\end{itemize}
By summing all the contributes we get
\begin{equation}
\begin{split}
\tau_{\ell-1}= -\ell+3+i_{j-1}+2(S_2(j)-T_1(j)),
\end{split}
\end{equation}
which proves \eqref{eq: dnformula5}.
\end{itemize}

\item Case $i_j=\ell-1,\ell$. Without loss of generality we can assume $i_j=\ell-1$. For this case we have to compute
\begin{equation}
\tau_{i_j}=-\delta^c_{\ell-2}+2\delta^c_{\ell-1}.
\end{equation}

\begin{itemize}
\item Case $|S\cap \lrbr{\gamma_{\ell-1},\gamma_{\ell}}|=1$. We compute
\begin{equation}\label{eq: dnformula6}
\tau_{i_j}=2(S_2(j) -T_1(j)),
\end{equation}
that, together with the coefficients \eqref{eq: gammadn} and \eqref{eq: xiCoefficientsDnPre}, implies identity \eqref{eq: xi_Dn6}. Similarly to the previous point, the roots that contribute in the summation are 
\begin{itemize}
\item $\sum_{i=n_1}^{\ell-2}\gamma_i,\ |S\cap\lrbr{n_1,\ldots,\ell-2}|\in 2\ZZ$, which contribute by $-\sum_{k=1}^{\lfloor\frac{j+1}{2} \rfloor-1} i_{j-2k}-i_{j-2k-1}$.
\item $\sum_{i=n_1}^{\ell-2}\gamma_i+\gamma_{\ell-1},\ |S\cap\lrbr{n_1,\ldots,\ell-2}|\in 1+2\ZZ$, which contribute by $\sum_{k=1}^{\lfloor\frac{j}{2} \rfloor} i_{j-2k+1}-i_{j-2k}$.
\item $\sum_{i=n_1}^{\ell-2}\gamma_i+\gamma_{\ell},\ |S\cap\lrbr{n_1,\ldots,\ell-2}|\in 2\ZZ$, which contribute by $-\sum_{k=1}^{\lfloor\frac{j+1}{2} \rfloor-1} i_{j-2k}-i_{j-2k-1}$.
\item $\sum_{i=n_1}^{\ell-2}\gamma_i+\gamma_{\ell-1}+\gamma_{\ell},\ |S\cap\lrbr{n_1,\ldots,\ell-2}|\in 1+2\ZZ$ which contribute by $\sum_{k=1}^{\lfloor\frac{j}{2} \rfloor} i_{j-2k+1}-i_{j-2k}$.
\end{itemize}
Thus one gets
\begin{equation}
\tau_{i_j}=2(S_2(j)-T_1(j)),
\end{equation}
which proves \eqref{eq: dnformula6}

\item Case $|S\cap \lrbr{\gamma_{\ell-1},\gamma_{\ell}}|=2$. We prove
\begin{equation}\label{eq: dnformula7}
\tau_{\ell-1}=\ell-2-i_{m-3},
\end{equation}
that, together with the coefficients \eqref{eq: gammadn} and \eqref{eq: xiCoefficientsDnPre}, implies identity \eqref{eq: xi_Dn7}. Similarly as above one can compute the roots that contribute to $\tau_{i_j}$, except that one has to take into account that the compactness of the roots changes. Then
\begin{equation}
\tau_{\ell-1}=\ell-2-i_{m-3},
\end{equation}
which proves identity \eqref{eq: dnformula7} and concludes the proof.
\end{itemize}
\end{itemize}
\end{proof}

\paragraph{Aknowledgments} The author is grateful to Alberto Della Vedova for discussions, ideas and thoughts in the initial stage of the work, and to G. Bruno De Luca for many corrections, suggestions and for the support in writing this paper.

\bibliographystyle{siam}
\bibliography{biblio}
\vspace{0.5cm}
Alice Gatti\\
\textit{Applied Mathematics and Computational Research Division, Lawrence Berkeley National Laboratory, 1 Cyclotron Road, Berkeley, CA 94720}.\\ 
E-mail: \texttt{agatti@lbl.gov}
\end{document}